\documentclass[reqno, 11pt, a4paper]{amsart}
\usepackage{latexsym,ifthen,xspace}
\usepackage{amsmath,amssymb,amsthm}
\usepackage{enumerate, calc, bbm}
\usepackage[usenames,dvipsnames]{xcolor}
\usepackage{paralist}
\usepackage[colorlinks=true, pdfstartview=FitV, linkcolor=blue,
  citecolor=blue, urlcolor=blue, pagebackref=false]{hyperref}
\usepackage[text={33pc,605pt},centering]{geometry}

\allowdisplaybreaks

\newtheorem{theorem}{Theorem}[section]
\newtheorem{lemma}[theorem]{Lemma}
\newtheorem{prop}[theorem]{Proposition}
\newtheorem{assumption}[theorem]{Assumption}

\theoremstyle{definition}

\theoremstyle{remark}
\newtheorem{remark}[theorem]{Remark}

\numberwithin{equation}{section}

\DeclareMathAlphabet{\mathsl}{OT1}{cmss}{m}{sl}
\SetMathAlphabet{\mathsl}{bold}{OT1}{cmss}{bx}{sl}

\newcommand{\cF}{\ensuremath{\mathcal F}}

\newcommand{\cL}{\ensuremath{\mathcal L}}

\newcommand{\bbN}{\ensuremath{\mathbb N}}

\newcommand{\bbP}{\ensuremath{\mathbb P}}

\newcommand{\bbR}{\ensuremath{\mathbb R}}

\newcommand{\bbZ}{\ensuremath{\mathbb Z}}

\DeclareMathOperator{\mean}{\mathbf{E}}

\DeclareMathOperator{\prob}{\mathbf{P}} 

\newcommand{\ldef}{\ensuremath{\mathrel{\mathop:}=}}
\newcommand{\rdef}{\ensuremath{=\mathrel{\mathop:}}}

\newcommand{\indicator}{\mathbbm{1}}

\begin{document}

\title[Heat kernel fluctuations and quantitative homogenization for the BTM]
{Heat kernel fluctuations and quantitative homogenization for the one-dimensional Bouchaud trap model}

\author{Sebastian Andres}
\address{Technische Universit\"at Braunschweig}
\curraddr{Institut f\"ur Mathematische Stochastik, Universit\"atsplatz 2, 38106 Braunschweig, Germany}
\email{sebastian.andres@tu-braunschweig.de}
\thanks{}

\author{David A.\ Croydon}
\address{Research Institute for Mathematical Sciences}
\curraddr{Kyoto University, Kyoto 606-8502, Japan}
\email{croydon@kurims.kyoto-u.ac.jp}

\author{Takashi Kumagai}
\address{Waseda University}
\curraddr{3-4-1 Okubo, Shinjuku-ku, Tokyo 169-8555, Japan}
\email{t-kumagai@waseda.jp}

\thanks{}



\date{\today}

\dedicatory{}

\begin{abstract}
We present on-diagonal heat kernel estimates and quantitative homogenization statements for the one-dimensional Bouchaud trap model. The heat kernel estimates are obtained using standard techniques, with key inputs coming from a careful analysis of the volume growth of the invariant measure of the process under study. As for the quantitative homogenization results, these include both quenched and annealed Berry-Esseen-type theorems, as well as a quantitative quenched local limit theorem. Whilst the model we study here is a particularly simple example of a random walk in a random environment, we believe the roadmap we provide for establishing the latter result in particular will be useful for deriving quantitative local limit theorems in other, more challenging, settings.
\end{abstract}

\maketitle

\section{Introduction}

A principal motivation for studying random walks in random environments is to explore the relationships that might exist between the microscopic properties of a disordered medium and its large scale conductivity. One particular question of interest is whether homogenization occurs, in the sense that the deformities in the random medium are `smoothed out', so that only an averaged effect remains; the alternative being that the local inhomogeneities persist in the scaling limit and some kind of anomalous behaviour is seen.

In the case of the one-dimensional Bouchaud trap model (as will be introduced precisely in Section~\ref{model} below), it is known that there is a phase transition between a regime where homogenization occurs, and one where it does not. The first goal of this article is to highlight the differences between the behaviour of the heat kernel in these two regimes, demonstrating how one sees a local limit theorem in the former, and almost-sure heat kernel fluctuations that are closely linked to the inhomogeneity of the associated invariant measure in the latter (see Section~\ref{hkf}). In the case of homogenization, it is natural to ask how quickly this occurs. Accordingly, as our second main contribution, we derive a number of quantitative homogenization results for the one-dimensional Bouchaud trap model (see Section~\ref{qsh} for details). These include a Berry-Esseen-type theorem, which quantifies the rate at which the marginals of the discrete model converge to a normal distribution. Moreover, combining such statements with our volume estimates, we are able to establish a quantitative quenched local limit theorem. We expect the somewhat basic analytic approach we adopt for arriving at this conclusion is quite robust and should be adaptable to other settings.

\subsection{The model}\label{model}

The \emph{Bouchaud trap model} (BTM) is a random walk in a random medium given by a landscape of traps which retain the walk for some amount of time. It has its origins in the statistical physics literature, where it was proposed as a simple effective model for the dynamics of spin-glasses on certain time-scales, see \cite{Bou92}. The one-dimensional versions of the model considered in this paper were first studied in \cite{FIN99}; for a general overview of the BTM we refer to \cite{BC06}. A distinctive feature of the BTM is that when the distribution of the traps is  sufficiently heavy-tailed, the trapping mechanism becomes relevant even in the large scale behaviour of the system, and the model exhibits localization \cite{CM15, CM16, CM17, FIN99, Mu15}, subdiffusivity \cite{Ca15,FIN} and aging \cite{BC05, Ce06,FIN}.

To describe the model in detail, we start by defining a trapping landscape $\tau=(\tau_x)_{x\in \bbZ}$, which is a collection of independent and identically distributed (i.i.d.)\ random variables, built on a probability space with
probability measure $\prob$, such that
\begin{equation}\label{tautail}
\prob[\tau_0 \geq u]= u^{-\alpha}, \qquad \forall u\geq 1,
\end{equation}
for some $\alpha \in (0,\infty)$. (We expect that most of our results will hold for wider classes of distributions that are bounded below and have similar tail asymptotics, but since our focus is on highlighting the possible range of behaviours for the BTM, rather than giving the most general statements possible, we have chosen to restrict to this particular family of distributions, which simplifies various calculations.) Conditional on $\tau$, the dynamics of the (symmetric) BTM are given by a continuous-time $\bbZ$-valued Markov chain $X=(X_t)_{t\geq 0}$,  with generator
\begin{align}\label{eq:geneBTM}
\cL f(x) = \frac{1}{2\tau_x} \sum_{y\sim x} \big(f(y)-f(x) \big).
\end{align}
Thus the dynamics of $X$ are particularly simple. Namely, the process $X$ waits at a vertex $x$ an exponentially distributed time with mean $\tau_x$; after this time, it jumps to one of the two neighbours of $x$ chosen uniformly at random. It is easy to see that the measure $\sum_{x\in \bbZ} \tau_x \delta_x$ is a reversible measure for the Markov chain $X$.

In the following, we denote by $\mean$ the expectation operator associated with the probability measure $\prob$ determining the law of $\tau$. Given a particular realisation of $\tau$, we write $P^\tau_x$ for the law of $X$, when started from $x$; this is the quenched law of $X$. 
The quenched transition density (or heat kernel) is given by
\begin{equation}\label{qtd}
p^\tau_t(x,y):=\frac{P^\tau_x(X_t=y)}{\tau_y}.
\end{equation}
Note that the normalization by $\tau_y$ is natural, since $\tau$ is a stationary measure for the reversible Markov chain, and so $p^\tau(x,y)$ is a symmetric function. (For our asymptotic results, the particular normalization is unimportant, apart from the precise scaling statements of Theorems \ref{thm:QHK}(iii) and \ref{thm:QLCLT}.) Moreover, we write
\[\mathbb{P}(\cdot)=\int P_0^\tau(\cdot)\mathbf{P}(d\tau)\]
for the annealed law of the process $X$ started from 0.

\subsection{Scaling limits and heat kernel fluctuations}\label{hkf}

A well-known phenomenon in the study of stochastic processes in random media is the occurrence of anomalous heat kernel behaviour when the medium is highly inhomogeneous. Possible manifestations of such irregular behaviour are the presence of fluctuations for the short- or long-time asymptotics of the on-diagonal heat kernel, or scaling towards non-Gaussian processes under non-diffusive space-time scaling. Some examples include the random conductance model, Brownian motion on certain random fractals, and simple random walks on the incipient infinite cluster of a critical Galton-Watson tree or on two-dimensional uniform spanning trees, see \cite{ACK} for a recent survey about the occurrence of heat kernel fluctuations. In this paper, we add the one-dimensional symmetric Bouchaud trap model to this class of examples. (We note that the results of this section were announced in \cite[Section~12.5]{ACK}.)

We first recall the scaling limits known to hold for the process $X$. In each case in the statement of the subsequent theorem, the topology considered is the usual Skorohod $J_1$ topology on the space of c\`{a}dl\`{a}g paths $D([0,\infty))$.

\begin{theorem}\label{thm:SL}
\hspace{10pt}
\begin{enumerate}
\item[(i)] If $\alpha\in(0,1)$, then $(\varepsilon X_{t/\varepsilon^{1+1/\alpha}})_{t\geq 0}$ converges in distribution under the annealed law $\mathbb{P}$ to the Fontes-Isopi-Newman diffusion of \cite[Definition~1.2]{FIN} with parameter $\alpha$.
\item[(ii)] If $\alpha=1$, then $(\varepsilon X_{t\log(\varepsilon^{-1})/\varepsilon^{2}})_{t\geq 0}$ converges in distribution under the annealed law $\mathbb{P}$ to Brownian motion.
\item[(iii)] If $\alpha>1$, then, for $\mathbf{P}$-a.e.\ realisation of the environment $\tau$, $(\varepsilon X_{t\mathbf{E}(\tau_0)/\varepsilon^{2}})_{t\geq 0}$
converges in distribution under the quenched law $P^\tau_0$ to Brownian motion.
\end{enumerate}
\end{theorem}
\begin{proof}
Part~(i) is contained within \cite[Theorem~3.2(ii)]{BCCR}. (We note that all the claims of \cite[Theorem~3.2]{BCCR} readily extend to the case $\beta=0$.) See also \cite{FIN} for the introduction of the limiting process and a scaling limit for a single marginal of the process, and \cite{BC05} for a related result in the non-symmetric case. Part~(iii) is contained within  \cite[Theorem~3.2(i)]{BCCR}. Part~(ii) is described in \cite[Remark~3.3]{BCCR}, and can also be checked by applying the resistance scaling techniques of \cite{Croyres,CHKres}. Indeed, via such an approach, a similar result was checked in the more challenging case of the random walk on the range of random walk in four dimensions in \cite{Cshir}.
\end{proof}

The above result shows that the BTM exhibits quenched homogenization when $\alpha>1$ and anomalous scaling behaviour for $\alpha\leq 1$. Our goal here is to show that in the latter case, the on-diagonal part of the heat kernel fluctuates around a leading term of polynomial order, and in the former case, the heat kernel satisfies a quenched local limit theorem. To capture the natural scaling of the on-diagonal heat kernel, we introduce the following $\alpha$-dependent functions:
\[\phi_\alpha(t):=\left\{
                    \begin{array}{ll}
                      t^{-\frac{1}{1+\alpha}}, & \hbox{if $\alpha\in(0,1)$;} \\
                      t^{-\frac{1}{2}}(\log t)^{-\frac{1}{2}}, & \hbox{if $\alpha=1$;} \\
                      t^{-\frac12}, & \hbox{if $\alpha>1$.}
                    \end{array}
                  \right.\]
Specifically, we are able to prove the following quenched heat kernel estimates. Note that the exponents in the limsup part of (i) match those for the limiting Fontes-Isopi-Newman diffusion, as verified in \cite[Theorem~4.6]{CHK}. Similarly, the liminf part of (i) essentially matches the result of \cite[Theorem~4.5]{CHK}, modulo a more careful computation of log exponents that is possible in our simpler setting. (In \cite{ACK}, part~(iii) is presented slightly differently, and erroneously misses the factor of $\mathbf{E}(\tau_0)$ that arises in the scaling of the invariant measure.)

\begin{theorem}\label{thm:QHK}
\hspace{10pt}
\begin{enumerate}
\item[(i)] If $\alpha\in(0,1)$, then it $\prob$-a.s.\ holds that
\[\limsup_{t\rightarrow\infty}\frac{p_t^\tau(0,0)}{\phi_\alpha(t)(\log\log t)^{\frac{1-\alpha}{1+\alpha}}}\in(0,\infty),\]
and
\[\liminf_{t\rightarrow\infty}\frac{p_t^\tau(0,0)}{\phi_\alpha(t)(\log t)^{-\theta}}=\begin{cases}
0,  & \text{if $\theta \leq\frac{1}{1+\alpha}$},\\
\infty, & \text{if $\theta>\frac{3}{\alpha}$.}
\end{cases}\]
\item[(ii)] If $\alpha=1$, then it $\prob$-a.s.\ holds that
\[\limsup_{t\rightarrow\infty}\frac{p_t^\tau(0,0)}{\phi_\alpha(t)}\in(0,\infty),\]
and
\[\liminf_{t\rightarrow\infty}\frac{p_t^\tau(0,0)}{\phi_\alpha(t)(\log\log t)^{-\theta}}=\begin{cases}
0,  & \text{if $\theta=\frac12$},\\
\infty, & \text{if $\theta>3$.}
\end{cases}\]
\item[(iii)] If $\alpha>1$, then it $\prob$-a.s.\ holds that, for every $x_0\in[0,\infty)$ and compact interval $I\subseteq(0,\infty)$,
    \[\lim_{\lambda\rightarrow\infty}\sup_{|x|\leq x_0}\sup_{t\in I}\left|\mathbf{E}(\tau_0)\lambda p^\tau_{\mathbf{E}(\tau_0)\lambda^2t}\left(0,\lfloor \lambda x\rfloor\right)-\frac{1}{\sqrt{2\pi t}}e^{-\frac{x^2}{2 t}} \right|=0.\]
    In particular,
    \[\lim_{t\rightarrow\infty}\phi_\alpha(\mathbf{E}(\tau_0)t)^{-1}p_t^\tau(0,0)=\frac{1}{\sqrt{2\pi}}.\]
\end{enumerate}
\end{theorem}

\begin{remark}
(i) We anticipate that the various $\limsup_{t\to\infty}$ and $\liminf_{t\to\infty}$ appearing in the above result are deterministic (i.e. independent of the environment $\tau$), though our arguments do not yield this. \\  
(ii)
Similarly to \cite[Corollary 1.4]{Cshir}, for $\alpha=1$, it would be possible to prove a local limit theorem in probability (i.e.\ a statement analogous to part (iii) above, but with the limit holding in probability, and appropriate logarithmic factors appearing in the scaling).
\end{remark}

We conclude the section with asymptotic statements concerning the distribution and expectation of the heat kernel that justify the choice of $\phi_\alpha$ as the natural scaling function. We note that part (ii) of the result is similar to that of \cite{CKtree}, in that we are only able to check the appropriate integrability of the heat kernel for a certain range of the parameter $\alpha$.

\begin{theorem}\label{thm:AHK}
\hspace{10pt}
\begin{enumerate}
\item[(i)] For every $\alpha>0$,
\[\lim_{\lambda\rightarrow\infty}\liminf_{t\rightarrow\infty}
\prob\left(\lambda^{-1}\leq \phi_\alpha(t)^{-1}p_t^\tau(0,0)\leq \lambda\right)=1.\]
\item[(ii)] For every $\alpha>0$, there exists an $\varepsilon>0$ such that
\[0<\liminf_{t\rightarrow\infty} \phi_\alpha(t)^{-\varepsilon}\mean\left(p_t^\tau(0,0)^\varepsilon\right)
\leq \limsup_{t\rightarrow\infty}  \phi_\alpha(t)^{-\varepsilon}\mean\left(p_t^\tau(0,0)^\varepsilon\right)<\infty.\]
Moreover, if $\alpha>\frac{3}{2}$, then we may take $\varepsilon=1$.
\end{enumerate}
\end{theorem}

\subsection{Quantitative stochastic homogenization}\label{qsh}

During the last decade, starting from the works \cite{GO11,GO12}, the area of quantitative stochastic homogenization of elliptic and parabolic equations has witnessed tremendous progress, see e.g.\ \cite{AG18, armstrong2016b, armstrong2016, armstrong2014,BellaFFOtto2016,BellaFO2016,DaGu,DGO20,GNO15,GNOreg,GO15,GO14}, and the two monographs \cite{AK22,AKM19}. Closely related to the topic of homogenization in PDE theory is the problem of deriving invariance principles for the associated random walk in random environment. For the one-dimensional BTM with $\alpha>1$, the results stated in the last subsection include an invariance principle in Theorem~\ref{thm:SL}(iii) and a local limit theorem in Theorem~\ref{thm:QHK}(iii). As noted at the start of the introduction, it is natural to seek to strengthen such statements and to provide some corresponding quantitative stochastic homogenization results. In the case when $\alpha>2$, that is when the variance of $\tau_x$ is finite, we are able give some estimates on the speed of convergence in the invariance principle and the local limit theorem. We first quantify the speed of convergence in the invariance principle  by means of quenched and annealed Berry-Esseen-type estimates. Throughout the article, we use the $O$ and $o$ notation to describe asymptotic bounds (as $t\rightarrow \infty$) in the usual way.

\begin{theorem}[Berry-Esseen theorem] \label{thm:be}
Suppose that $\alpha>2$ and let $\sigma^2 \ldef 1/ \mean[\tau_0]$.
\begin{enumerate}
\item[(i)] It holds that
\[  \sup_{x \in \bbR} \Big|  \bbP \big[X_t \leq \sigma x \sqrt{t}  \ \big] - \Phi(x)\Big|  \; = \; O\left(t^{- \frac 1 {10}} (\log t)^{\frac 1 5}\right),\]
where $\Phi(x):=(2\pi)^{-1/2} \int_{-\infty}^x e^{-u^2/2} \, du$ denotes the distribution function of the standard normal distribution.
\item[(ii)]  Fix $\theta\in (0, \frac{\alpha-2}{10(2\alpha-1)})$. For $\prob$-a.e.\ $\tau$, it holds that
\[\sup_{x \in \bbR} \Big|   P_0^\tau \big[X_t \leq \sigma x \sqrt{t} \ \big]  - \Phi(x)\Big| =O\left(t^{-\theta}\right).\]
\end{enumerate}
\end{theorem}

The annealed bound in Theorem~\ref{thm:be}(i) matches (up to the logarithmic correction) the bound obtained in \cite{Mo12} for the random conductance model, another model of a random walk in random environment, which has some strong connections with the BTM. More precisely, for the random conductance model with uniformly elliptic i.i.d.\ conductances, an annealed Berry-Esseen theorem as in Theorem~\ref{thm:be}(i) was proven in \cite{Mo12} for arbitrary dimension $d\geq 1$ with rate $t^{-1/10}$ (plus logarithmic corrections for $d=2$), and with rate $t^{-1/5}$ for $d\geq 3$ (with some logarithmic corrections for $d=3$). In \cite{AN19}, this was extended to some quenched and annealed bounds in $d\geq 3$ for unbounded and correlated conductances satisfying strong moment conditions and a quantified ergodicity assumption in the form of a spectral gap estimate.

Given that the random walk $(X_t)_{t\geq 0}$ is a martingale, our approach to proving Theorem~\ref{thm:be} is based on a Berry-Esseen-type estimate for martingales of \cite{Ha88} (see Theorem~\ref{thm:be_hb} below). It requires quantification of the speed of convergence of
\[E_0^\tau \left[\Big| \frac{\langle X \rangle_t} t -\sigma^2 \Big|^2 \right]\]
towards zero, where $\langle X \rangle$ denotes the quadratic variation process of the martingale $X$, see Proposition~\ref{prop:estV} below. To this end, we exploit the fact that $X$ is a time-change of the simple random walk on $\bbZ$. Then it suffices to show similar quantitative estimates on the additive functional governing the time-change, which is a continuous time version of the random walk in random scenery studied in \cite{DF19}.

\begin{remark}
The exponents  in Theorem~\ref{thm:be} are non-optimal. On the one hand, this is due to our non-optimal (quenched) estimate for the random walk in random scenery of Proposition~\ref{prop4-3}. On the other hand, an additional factor of  $1/5$ in the exponent arises from the application of the general Berry-Esseen theorem for martingales (see Theorem~\ref{thm:be_hb} below with the choice $n=2$). Recall that the jumps of the martingale $X$ are obviously bounded. In such a situation the results in \cite{Mo12a} show that the decay rate given by $E_0^\tau [|\frac{\langle X \rangle_t} t -\sigma^2 |^2]^{1/5}$ is optimal (for the choice $n=2$ in Theorem~\ref{thm:be_hb}). Thus, the only possibility to improve the exponents (within this approach) is to apply Theorem~\ref{thm:be_hb} for larger values of $n$, which requires control on higher moments of $\big| \frac{\langle X \rangle_t} t -\sigma^2 \big|$ (cf.\ the discussions of \cite[page 6]{Mo12} and \cite[Remark~1.12]{AN19}). Note that the rate $t^{-1/2}$ is the best possible one, since it is known to be the convergence rate for the simple random walk.
\end{remark}

By combining the Berry-Esseen estimates with suitable volume estimates we derive the following quantitative version of the quenched local limit theorem.

\begin{theorem} [Quantitative local limit theorem]\label{thm:QLCLT}
Suppose that $\alpha>2$. For any $\theta\in (0, \frac{\alpha-2}{10(2\alpha-1)})$, for all $K>0$ and $0<T_1\leq T_2$,
\[\lim_{n\to  \infty} n^{2\theta/3} \sup_{|x|\leq K} \sup_{t\in[T_1,T_2]}\bigg|\mean[\tau_0] \, n \, p^\tau_{\mean[\tau_0]n^2t}(0,\lfloor nx \rfloor)-\frac{1}{\sqrt{2\pi t}}e^{-\frac{x^2}{2 t}} \bigg| =0,  \qquad \mathbf{P}\text{-a.s.}\]
\end{theorem}

Again it would be desirable to identify and prove optimal rates of convergence in the local limit theorem. In the context of the random conductance model, optimal quantitative homogenization results for heat kernels and Green functions can be established by using techniques from quantitative stochastic homogenization, see \cite[Chapters~8--9]{AKM19} for details in the uniformly elliptic case with finite range dependencies. This technique has been adapted to random conductance models on i.i.d.\ percolation clusters in \cite{DaGu}, and it is expected that it also applies to other degenerate models.

\begin{remark}
Whilst we do not pursue it here, it would be of further interest to prove an annealed version of Theorem~\ref{thm:QLCLT}. Furthermore, we know that a quenched local limit theorem holds for all $\alpha>1$, see Theorem~\ref{thm:QHK}(iii), and so it is also relevant to ask about the speed of convergence when $\alpha\in (1,2]$. Since we use various second moment estimates, this is beyond the techniques of this article.
\end{remark}

\subsection{Organisation of the article} The remainder of the article is organised as follows. In Section~\ref{sec:vol}, we summarise classical results concerning sums of i.i.d.\  heavy-tailed random variables, which will shed light on the volume growth of the invariant measure in the BTM. (Related tail bounds are presented in Appendix~\ref{appa}.) In Section~\ref{sec:hke}, we prove the heat kernel estimates of Theorems~\ref{thm:QHK} and \ref{thm:AHK}. Finally, in Section~\ref{sec:qh}, we establish the quantitative homogenization results of Theorems~\ref{thm:be} and \ref{thm:QLCLT}.

\section{Law of iterated logarithm in the $\alpha$-stable regime}\label{sec:vol}

The goal of this section is to summarise the asymptotic fluctuations of i.i.d.\ sums with distributional tail as at \eqref{tautail}. In particular, from classical results, we will deduce the following theorem. Note that, for our heat kernel estimates, we only need the detailed statements of parts (i) and (ii), but we include the remaining cases for completeness. (It would potentially be interesting to determine whether parts (iii)-(v) could be used to understand the second order behaviour of the heat kernel in the relevant cases.) We use the notation $x\vee y$ and $x\wedge y$ for $\max\{x,y\}$ and $\min\{x,y\}$, respectively.

\begin{theorem} \label{thm:LIL}
Let $X_1, X_2, \ldots$ be i.i.d.\ on some probability space $(\Omega, \cF, \prob)$ with distribution function $F$ of the form  $F(x)=(1- c_F x^{-\alpha}) \vee 0$ with $c_F\in (0,\infty)$ and $\alpha \in (0,\infty)$. Write $S_n\ldef \sum_{i=1}^n X_i$. Then the following hold.
\begin{enumerate}
\item[(i)] If $\alpha\in (0,1)$, then $\prob$-a.s.
\begin{align} \label{eq:limsup_alpha01}
\limsup_{n\to\infty} \frac{S_n}{ n^{1/\alpha}  (\log n)^{1/\alpha} (\log\log n)^{\theta}}
=
\begin{cases}
0,  & \text{if $\theta>1/\alpha$},\\
\infty, & \text{if $\theta\leq 1/\alpha$,}
\end{cases}
\end{align}
and
\begin{align} \label{eq:liminf_alpha01}
\liminf_{n\to\infty} \frac{S_n}{n^{1/\alpha} (\log\log n)^{1-1/\alpha}}
=   c_F^{1/\alpha} \frac \alpha {1-\alpha}  \Gamma(2-\alpha)^{1/\alpha}.
\end{align}
\item[(ii)] If $\alpha=1$, then $\prob$-a.s.
\begin{align} \label{eq:limsup_alpha1}
\limsup_{n\to\infty} \frac{S_n}{ n  (\log n) (\log\log n)^{\theta}}
=
\begin{cases}
0,  & \text{if $\theta>1$},\\
\infty, & \text{if $\theta\leq 1$,}
\end{cases}
\end{align}
and
\begin{align} \label{eq:liminf_alpha1}
\liminf_{n\to \infty} \frac{S_n}{n\log n} =c_F.
\end{align}
\item[(iii)] If $\alpha\in(1,2)$, then $\prob$-a.s.
\begin{align} \label{eq:limsup_alpha12}
\limsup_{n\to\infty} \frac{S_n-n \mean[X_1]}{ n^{1/\alpha}  (\log n)^{1/\alpha} (\log\log n)^{\theta}}
=
\begin{cases}
0,  & \text{if $\theta>1/\alpha$},\\
\infty, & \text{if $\theta\leq 1/\alpha$,}
\end{cases}
\end{align}
and
\begin{align} \label{eq:liminf_alpha12}
\liminf_{n\to\infty} \frac{S_n- n \mean[X_1]}{n^{1/\alpha} (\log\log n)^{1-1/\alpha}} 
= c_F^{1/\alpha}\frac\alpha{1-\alpha}\big( \Gamma(2-\alpha)^{1/\alpha}-1\big).
\end{align}
\item[(iv)] If $\alpha=2$, then $\prob$-a.s.
\begin{align} \label{eq:limsup_alpha2}
\limsup_{n\to\infty} \frac{S_n- n \mean[X_1]}{\sqrt{ n  (\log n) (\log\log n)^{\theta}}}
=
\begin{cases}
0,  & \text{if $\theta>1$},\\
\infty, & \text{if $\theta\leq 1$,}
\end{cases}
\end{align}
and
\begin{align} \label{eq:liminf_alpha2}
\liminf_{n\to \infty} \frac{S_n- n \mean[X_1]}{\sqrt{n(\log n)(\log\log n)}} =-\sqrt{2c_F}.
\end{align}
\item[(v)] If $\alpha \in (2,\infty)$, then $\prob$-a.s.
\[\limsup_{n\to \infty} \frac{S_n- n \mean[X_1]}{\sqrt{n \log\log n}} = \sqrt{2 \mathrm{Var}(X_1)}, \quad \liminf_{n\to \infty} \frac{S_n- n \mean[X_1]}{\sqrt{n \log\log n}} = -\sqrt{2 \mathrm{Var}(X_1)}.\]
\end{enumerate}
\end{theorem}

Towards proving Theorem~\ref{thm:LIL}, we first analyse the median of $S_n$ for large $n$ in the case when $\alpha<2$.

\begin{lemma} \label{lem:median}
For $n\geq 1$, let  $m_n$ denote the median of $S_n$. Then, as $n\to \infty$,
\begin{align*}
m_n =
\begin{cases}
\bar m_\alpha  n^{1/\alpha}+o( n^{1/\alpha}), & \text{if $\alpha \in (0,1)$}, \\
\bar m_\alpha  n \log n+o(n\log n), & \text{if $\alpha=1$}, \\
n \mean[X_1] + \bar m_\alpha  n^{1/\alpha}+ o(n^{1/\alpha}), & \text{if $\alpha \in (1,2)$},
\end{cases}
\end{align*}
with  $\bar m_\alpha$ denoting the median of a totally asymmetric $\alpha$-stable law $S_1^{(\alpha)}$ satisfying
\begin{equation}\label{slaplace}
\mean\left[e^{-\lambda S_1^{(\alpha)}}\right]  =e^{-C\lambda^\alpha}
\end{equation}
for some appropriately chosen $C\in (0,\infty)$ (depending on $c_F$ and $\alpha$).
\end{lemma}

\begin{proof}
Recall that for $\alpha \in (1,2)$,
\begin{align*}
\frac{S_n - n \mean[X_1]}{n^{1/\alpha}} \overset{d}{\longrightarrow} S_1^{(\alpha)} \qquad \text{as $n\to \infty$,}
\end{align*}
where the law of $S_1^{(\alpha)}$ is characterized by \eqref{slaplace}, see \cite[Chapter~7]{GK68} or \cite[Chapter~8]{BGT}, for example. In particular, for any $\varepsilon>0$,
\begin{align*}
\prob\big[ S_n \leq n \mean[X_1] + (\bar m_\alpha - \varepsilon) n^{1/\alpha} \big] &= \prob \Big[ \frac{S_n - n \mean[X_1]}{n^{1/\alpha}} \leq \bar m_\alpha - \varepsilon  \Big] \\
&
\underset{n\to \infty}{\longrightarrow} \prob\big[ S_1^{(\alpha)}  \leq \bar m_\alpha - \varepsilon \big] < \frac  1 2,
\end{align*}
and similarly
\begin{align*}
\prob\big[ S_n \geq n \mean[X_1] + (\bar m_\alpha + \varepsilon) n^{1/\alpha} \big] &= \prob \Big[ \frac{S_n - n \mean[X_1]}{n^{1/\alpha}} \geq \bar m_\alpha + \varepsilon  \Big] \\
&\underset{n\to \infty}{\longrightarrow} \prob\big[ S_1^{(\alpha)}  \geq \bar m_\alpha + \varepsilon \big] <\frac  1 2.
\end{align*}
This gives the result for $\alpha \in (1,2)$. In the other cases, the claim follows by the same argument since $S_n/n^{1/\alpha} \overset{d}{\longrightarrow} S_1^{(\alpha)}$ if $\alpha\in (0,1)$ and  $S_n/(n \log n) \overset{d}{\longrightarrow} S_1^{(1)}$ if $\alpha=1$.
\end{proof}

\begin{proof}[Proof of Theorem~\ref{thm:LIL}]
{\bf Case 1: $\alpha \in (0,2)$}. \emph{Upper bounds.} We first summarize some consequences of the results in \cite[Theorem~6.1]{Pr81}. Using the notation therein, we have in our setting
\[G_+(u) \ldef \prob[X_1 >u] \asymp u^{-\alpha}, \qquad G(u)= \prob[|X_1| >u] \asymp u^{-\alpha},\]
and
\[K(u) \ldef u^{-2} \int_{|y| \leq u} y^2 \, dF(y) = c_F  u^{-2} \int_1^u \alpha y^{1-\alpha} \, dy\asymp u^{-\alpha}.\]
In particular,
\[\liminf_{u\to \infty} \frac{G_+(u)}{G(u)+K(u)}>0.\]
Hence, by \cite[Theorem~6.1]{Pr81} we have $\prob$-a.s.,
\begin{align*}
\limsup_{n\to\infty} \frac{S_n-m_n}{\beta_n}
=\begin{cases}
0,  & \text{if } \sum_n \prob[X_1>\beta_n]<\infty, \\
\infty, & \text{if } \sum_n \prob[X_1>\beta_n]=\infty,
\end{cases}
\end{align*}
where we again write $m_n$ for the median of $S_n$. Now, choosing
\[\beta_n= n^{1/\alpha}  (\log n)^{1/\alpha} (\log\log n)^{\theta},\]
so that
\begin{align*}
 \sum_n \prob[X_1>\beta_n] =c_F \sum_n n^{-1} (\log n)^{-1} (\log \log n)^{-\alpha \theta}
 \begin{cases}
 <\infty & \text{if $\theta>1/\alpha$}, \\
  =\infty & \text{if $\theta \leq 1/\alpha$,}
 \end{cases}
\end{align*}
we obtain that, for $\alpha\in (0,2)$,
\begin{align*}
\limsup_{n\to\infty} \frac{S_n-m_n}{ n^{1/\alpha}  (\log n)^{1/\alpha} (\log\log n)^{\theta}}
=\begin{cases}
0,  & \text{if $\theta>1/\alpha$},\\
\infty, & \text{if $\theta\leq 1/\alpha$.}
\end{cases}
\end{align*}
Thus, from Lemma~\ref{lem:median} we obtain \eqref{eq:limsup_alpha01}, \eqref{eq:limsup_alpha1} and \eqref{eq:limsup_alpha12}.

\emph{Lower bounds.}
Let $Q$ denote the quantile function associated with  $F(x)=(1- c_F x^{-\alpha}) \vee 0$, that is
\[Q(s) \ldef \inf\big\{ x: F(x)\geq s \big\} = c_F^{1/\alpha} \big(1-s\big)^{-1/\alpha}, \qquad 0 < s <1,\]
set $Q(0)=Q(0+)$, and define the truncated mean
\begin{align*}
\mu(s) &\ldef \int_0^{1-s} Q(u) \, du =  c_F^{1/\alpha} \int_0^{1-s} \big(1-u \big)^{-1/\alpha} \, du =  c_F^{1/\alpha} \int_s^1
 u^{-1/\alpha} \, du  \\
& =
 \begin{cases}
c_F^{1/\alpha} \frac{\alpha}{1-\alpha} \big[s^{1-1/\alpha} -1 \big], & \text{if $\alpha \in (0,2)\backslash\{1\}$}, \\
 - c_F \log(s),   & \text{if $\alpha =1$}.
 \end{cases}
 \end{align*}
For $\alpha \in (1,2)$, note that $\mu(s) \rightarrow\mean[X_1]$ as $s\to 0$. The truncated variance is given by
\begin{align*}
\sigma^2(s)&\ldef s Q(1-s)^2 + \int_0^{1-s} Q^2(u) \, du - \big(sQ(1-s)+\mu(s) \big)^2 \\
&=c_F^{2/\alpha} \Big( s^{1-2/\alpha} + \frac{\alpha}{2-\alpha} s^{1-2/\alpha} - \big( s^{1-1/\alpha} +\mu(s)\big)^2 \Big) \\
& \sim c_F^{2/\alpha}\frac{2}{2-\alpha} s^{1-2/\alpha}.
\end{align*}
Now setting
\begin{align*}
b_n \ldef n^{-1} \log \log n, \qquad a_n\ldef n^{1/2} \sigma(b_n),
\end{align*}
we have by \cite[Theorem~1]{Ma94} that
\begin{align} \label{eq:mason}
\liminf_{n\to\infty} \frac{S_n-n \mu(b_n)}{a_n (\log\log n)^{1/2}} = K_\alpha, \qquad \text{$\prob$-a.s.}
\end{align}
where the constant $K_\alpha \in [-\sqrt{2}, 0]$ is given by
\begin{align*}
K_\alpha=
\begin{cases}
\big( \Gamma(2-\alpha)^{1/\alpha}-1\big) \alpha (1-\alpha)^{-1} \big((2-\alpha)/2\big)^{1/2}, &\text{if $\alpha \in (0,2)\backslash\{1\}$}, \\
-\gamma/\sqrt{2},  & \text{if $\alpha=1$,}
\end{cases}
\end{align*}
with $\gamma$ being the Euler constant, see \cite[Lemma~4]{Ma94}.

Now for $\alpha\in (0,1)$ we have
\begin{align*}
\lim_{n\to\infty} \frac{n \mu(b_n)}{a_n (\log\log n)^{1/2}} =
 \frac \alpha{1-\alpha} 
 \Big(\frac{2-\alpha}{2} \Big)^{1/2} \rdef C_\alpha,
\end{align*}
so that \eqref{eq:mason} can be rewritten as
\begin{align*}
\liminf_{n\to\infty} \frac{S_n}{a_n (\log\log n)^{1/2}} = C_\alpha+ K_\alpha, \qquad \text{$\prob$-a.s.}
\end{align*}
with
\begin{align*}
 C_\alpha+ K_\alpha= 
\frac \alpha{1-\alpha} \Big(\frac{2-\alpha}{2} \Big)^{1/2} \Gamma(2-\alpha)^{1/\alpha} >0.
\end{align*}
This is equivalent to \eqref{eq:liminf_alpha01}.

If $\alpha =1$,  then by \eqref{eq:mason},
\begin{align*}
\liminf_{n\to\infty} \frac{S_n-c_Fn \log n+c_Fn\log\log\log n}{c_F\sqrt{2} n}=
\liminf_{n\to\infty} \frac{S_n-n \mu(b_n)}{a_n (\log\log n)^{1/2}} = - \gamma/ \sqrt{2},
\end{align*}
and by dividing both sides by $\log n$ we derive \eqref{eq:liminf_alpha1}.

Finally, if $\alpha\in (1,2)$, we  have $\mu(b_n) \sim \mean[X_1]$, so that \eqref{eq:mason} becomes \eqref{eq:liminf_alpha12}.

\noindent
{\bf Case 2: $\alpha = 2$.} To check the statement at \eqref{eq:liminf_alpha2}, we again apply the results of \cite{Ma94}. In particular, defining notation as in the previous part of the proof, we have in this case that $\sigma^2(s)\sim -c_F\log(s)$ as $s\rightarrow 0$. Thus
$a_n\sim (c_F n \log(n))^{1/2}$, and so \cite[Theorem~1 and Lemma~4]{Ma94} yield a statement analogous to \eqref{eq:mason} with $K_2=-\sqrt{2}$. This is equivalent to the desired result.

For \eqref{eq:limsup_alpha2}, we first consider the problem for the symmetric distribution given by $\mathbf{P}(|X_1|\geq x)=x^{-2}$ for $x\geq 1$. In this case, it is explained in \cite[Section~8]{Fell} that, $\prob$-a.s.,
\[\limsup_{n\rightarrow\infty}\frac{S_n}{\sqrt{n(\log n)(\log\log n)}}=\infty.\]
From this (and the conclusion at \eqref{eq:liminf_alpha2}), the corresponding result for the case of positive $X_1$ readily follows. (One may alternatively deduce the result from \cite[Theorem~7.5]{Pr81}, which contains a convenient restatement of results from \cite{Kl1,Kl2}.) Moreover, for the case of positive $X_1$, it follows from \cite[Theorem~1]{EL} that, for any $\varepsilon>0$, $\prob$-a.s.,
\[\limsup_{n\rightarrow\infty}\frac{S_n}{\sqrt{n(\log n)(\log\log n)^{1+\varepsilon}}}=0.\]

\noindent
{\bf Case 3: $\alpha > 2$.}
Since the random variables have a finite variance in this case, the result follows from the classical Hartman-Wintner law of the iterated logarithm \cite{HW}.
\end{proof}

\section{Heat kernel estimates}\label{sec:hke}

In this section, we aim to establish the heat kernel estimates of Theorems~\ref{thm:QHK} and \ref{thm:AHK}. These are proved using what are now standard techniques in the area, see \cite{BCK,BJKS,BK06,Croy,Kum,KumSF, KM}, for example. The approaches of these papers depend on establishing good volume growth and resistance bounds. In our setting, the resistance is particularly simple. Indeed, if we consider the state space $\mathbb{Z}$ to be an electrical network with unit resistors on each edge, the resistance metric simply coincides with the Euclidean metric. Moreover, the simple structure of the space means that, if we write $B(x,n):=\{y\in\mathbb{Z}:\:|x-y|<n\}$ and the effective resistance as $R$, then from the parallel law we obtain $R(x, B(x,n)^c)=\frac n 2$. (See \cite{Barbook} for a definition of the effective resistance and background on the connections between random walks and effective resistance.) Consequently, to establish reasonable heat kernel estimates, it remains to study the behaviour of the invariant measure for the Markov chain $X$. Recalling that a reversible measure for $X$ is given by $\mathbf{\tau}\ldef \sum_{x\in \bbZ} \tau_x \delta_x$, we define
\begin{align*}
V(x,n) \ldef \sum_{y=x-n}^{x+n} \tau_y, \qquad x\in \bbZ, \, n\in \bbN,
\end{align*}
and write $V(n)\ldef V(0,n)$. Clearly, the fluctuations of $V(n)$ for large $n$ can be immediately deduced from Theorem~\ref{thm:LIL}. In what follows, we explain how to transfer those volume asymptotics to heat kernel estimates. For the quenched result of Theorem~\ref{thm:QHK}, we will apply the following bounds.

\begin{prop}\label{prop:hke}
\begin{enumerate}
  \item[(i)] Let $x\in\mathbb{Z}$. If $V(x,r)\geq V_l(r)$ for all $r\geq 0$, where $V_l:[0,\infty)\rightarrow[0,\infty)$ is a continuous, non-decreasing function that satisfies the doubling property (i.e.\ there exists a $C<\infty$ such that $V_l(2r)\leq CV_l(r)$), then there exists a constant $c<\infty$ such that
      \[p_t^\tau(x,x)\leq\frac{ch_l^{-1}(t)}{t}\]
      for all $t>0$, where $h_l(r):= rV_l(r)$.
  \item[(ii)] For every $x\in\mathbb{Z}$ and $n\in\mathbb{N}$, it holds that
  \[p^\tau_{2nV(x,n)}(x,x)\leq \frac{2}{V(x,n-1)}.\]
  \item[(iii)] For every $x\in\mathbb{Z}$ and $n\in\mathbb{N}$, it holds that
  \[p^\tau_{\frac{n}{2}V(x,n)}(x,x)\geq \frac{V(x,n)^2}{16V(x,2n)^3}.\]
\end{enumerate}
\end{prop}
\begin{proof} Part~(i) can be proved by the argument of \cite[Proposition~4.1]{Kum}, which in turn is based on the approach of \cite{BCK}. In particular, the proof of \cite[Proposition~4.1]{Kum} only uses a lower volume bound and the doubling property of the bounding function. See also the proof of \cite[Proposition~5]{Croy} for a similar argument.

Part~(ii) can be proved by an adaptation of the argument used to establish part (a). A version of this is stated as \cite[(6.11)]{Croy}, but with oversight of the fact that the volumes considered should be of the closed ball on the left-hand side, and the open ball on the right-hand side of the inequality.

As for part~(iii), we again apply a standard argument, versions of which appear in \cite[Section~4]{Kum} and \cite[Proposition~3.2]{KM}. This is particularly simple in our setting due to the specific form of the Green's function. Indeed, by an elementary Markov chain argument, it is possible to check that
\[E^\tau_y\left(T_{B(x,n)}\right)=\sum_{z}g_{B(x,n)}(y,z) \, \tau_z,\]
where $T_{B(x,n)}:=\inf\{t\geq 0:\:X_t\not\in B(x,n)\}$, $E^\tau_y$ is the expectation under $P_y^\tau$ and
\[g_{B(x,n)}(y,z)=\left\{
                    \begin{array}{ll}
                      \frac{(n-x+y)(n+x-z)}{2n}, & \hbox{if $x-n\leq y\leq z\leq x+n$;} \\
                      \frac{(n+x-y)(n-x+z)}{2n}, & \hbox{if $x-n\leq z\leq y\leq x+n$;} \\
                      0, & \hbox{otherwise.}
                    \end{array}
                  \right.\]
In particular, it holds that $g_{B(x,n)}(y,z)\leq \frac{n}{2}$ for any $y$ and $z$, and so
\[E^\tau_y\left(T_{B(x,n)}\right)\leq \frac{n}{2}V(x,n).\]
Similarly, since $g_{B(x,n)}(x,z)\geq \frac{n}{4}$ for $z\in[x-\frac{n}{2},z+\frac{n}{2}]$, it holds that
\[E^\tau_x\left(T_{B(x,n)}\right)\geq \frac{n}{4}V(x,\frac{n}{2}).\]
Consequently, applying the Markov property at time $t$, we find that
\begin{align*}
\frac{n}{2}V(x,n)&\leq E^\tau_x\left(T_{B(x,2n)}\right)\\
&\leq t+E^\tau_x\left((T_{B(x,2n)}-t)\indicator_{\{T_{B(x,2n)}>t\}}\right)\\
&\leq t+\sup_{y}E^\tau_y\left(T_{B(x,2n)}\right)P_x^\tau\left(T_{B(x,2n)}>t\right)\\
&\leq t+nV(x,2n)P_x^\tau\left(T_{B(x,2n)}>t\right),
\end{align*}
and rearranging yields
\[P_x^\tau\left(T_{B(x,2n)}>t\right)\geq \frac{nV(x,n)-2t}{2nV(x,2n)}.\]
To relate this inequality to the heat kernel, we apply the Cauchy-Schwarz inequality to obtain that
\begin{align*}
P_x^\tau\left(T_{B(x,2n)}>t\right)^2&\leq P_x^\tau\left(X_t\in B(x,2n)\right)^2\\
&= \left(\sum_{z\in B(x,2n)}p^\tau_t(x,z)\tau_z\right)^2\\
&\leq \sum_{z} p^\tau_t(x,z)^2\tau_z V(x,2n)\\
&= p^\tau_{2t}(x,x)V(x,2n).
\end{align*}
Finally, combining the two previous inequalities and setting $t=\frac{n}{4}V(x,n)$, we deduce the desired inequality.
\end{proof}

We are now ready to prove Theorems~\ref{thm:QHK} and \ref{thm:AHK}. In the following, we will write constants depending on the particular realisation of the environment $\tau$ as $c(\tau)$, for example; note that the value of such constants might change from line to line.

\begin{proof}[Proof of Theorem~\ref{thm:QHK}(i)] Recall that this part of the theorem applies to $\alpha\in(0,1)$. We start by checking the limsup part of the result. Firstly, by applying \eqref{eq:liminf_alpha01}, it is possible to deduce that, $\prob$-a.s., the assumptions of Proposition~\ref{prop:hke}(a) hold with $V_l(r):=c(\tau)r^{1/\alpha}(\log\log r)^{1-1/\alpha}$. (Strictly speaking, to handle small $r$, we should replace $r^{1/\alpha}(\log\log r)^{1-1/\alpha}$ by a function such as $\max\{1,r^{1/\alpha}(\log\log r)^{1-1/\alpha}\}$, but this does not affect the substance of what follows, and we choose to omit the truncation from the notation for brevity.) Noting that
\[h_l^{-1}(t)\sim C(\tau)t^{\frac{\alpha}{1+\alpha}}(\log\log t)^{\frac{1-\alpha}{1+\alpha}}\]
as $t\rightarrow\infty$, we thus obtain from Proposition~\ref{prop:hke}(a) that the limsup of the statement of Theorem~\ref{thm:QHK}(i) is finite. Secondly, observe that \eqref{eq:liminf_alpha01} further implies that there exists a subsequence $(n_i)_{i\geq 1}$ such that
\[c_1(\tau) \, n_i^{1/\alpha}\, (\log\log n_i)^{1-1/\alpha}\leq V(0,n_i)\leq V(0,2n_i)\leq c_2(\tau) \, n_i^{1/\alpha} \, (\log\log n_i)^{1-1/\alpha}.\]
Hence Proposition~\ref{prop:hke}(iii) yields that, along the subsequence $t_i=\frac{n_i}{2}V(0,n_i)$,
\[p_{t_i}^\tau(0,0)\geq \frac{c(\tau)}{n_i^{1/\alpha}(\log\log n_i)^{1-1/\alpha}}.\]
Further, since the sequence $(t_i)$ is bounded from above and below by constant multiples of $n_i^{1+1/\alpha}(\log\log n_i)^{1-1/\alpha}$, it follows that
\[p_{t_i}^\tau(0,0)\geq c(\tau) \, t_i^{-\frac{1}{1+\alpha}} \, (\log\log t_i)^{\frac{1-\alpha}{1+\alpha}},\]
which establishes that the limsup in the statement of Theorem~\ref{thm:QHK}(a) is strictly positive.

We next consider the liminf, starting with the upper bound. By \eqref{eq:limsup_alpha01}, for any $\varepsilon>0$, it $\prob$-a.s.\ holds that there exists a subsequence $(n_i)_{i\geq 1}$ such that
\begin{align*}
V(0,n_i-1)&\geq c(\tau) \, n_i^{1/\alpha} \, (\log n_i)^{1/\alpha}(\log\log n_i)^{1/\alpha}\\
V(0,n_i)&\leq c(\tau) \, n_i^{1/\alpha} \, (\log n_i)^{1/\alpha}(\log\log n_i)^{1/\alpha+\varepsilon}.
\end{align*}
Consequently, taking $t_i=2n_iV(0,n_i)$, we deduce from Proposition~\ref{prop:hke}(b) that
\[p_{t_i}^\tau(0,0)\leq \frac{c(\tau)}{n_i^{1/\alpha}(\log n_i)^{1/\alpha}(\log\log n_i)^{1/\alpha}}.\]
Since $t_i\leq c(\tau)n_i^{1+1/\alpha}(\log n_i)^{1/\alpha}(\log\log n_i)^{1/\alpha+\varepsilon}$, it follows that
\[p_{t_i}^\tau(0,0)\leq  c(\tau) \, t_i^{-\frac{1}{1+\alpha}}(\log t_i)^{-\frac{1}{1+\alpha}}\, (\log\log t_i)^{-\frac{1-\varepsilon}{1+\alpha}},\]
as desired. As for the lower bound on the liminf in the statement of Theorem~\ref{thm:QHK}(a), we apply both \eqref{eq:limsup_alpha01} and \eqref{eq:liminf_alpha01} to deduce that, for any $\varepsilon>0$, $\prob$-a.s., for all $n\in\mathbb{N}$,
\[c_1(\tau)\, n^{1/\alpha} \, (\log\log n)^{1-1/\alpha}\leq V(0,n-1)\leq V(0,2n)\leq c_2(\tau)\, n^{1/\alpha} \, (\log n)^{1/\alpha+\varepsilon}.\]
Consequently, if $t$ satisfies
\[\frac{n-1}{2}V(0,n-1)\leq t\leq \frac{n}{2}V(0,n),\]
then the monotonicity of the on-diagonal heat kernel (see \cite[Lemma~5.11]{Barbook}, for example) and Proposition~\ref{prop:hke}(iii) imply that
\[p_t^\tau(0,0)\geq p_{\frac{n}{2}V(0,n)}^\tau(0,0)\geq \frac{V(0,n)^2}{16V(0,2n)^3}\geq \frac{c(\tau)}{n^{1/\alpha}(\log n)^{3/\alpha+4\varepsilon}}\geq \frac{c(\tau)}{t^{\frac{1}{1+\alpha}}(\log t)^{\frac{3}{\alpha}+5\varepsilon}},\]
from which the result follows.
\end{proof}

\begin{proof}[Proof of Theorem~\ref{thm:QHK}(ii)] The proof of Theorem~\ref{thm:QHK}(ii) is essentially identical to that of Theorem~\ref{thm:QHK}(i), only replacing the estimates \eqref{eq:limsup_alpha01} and \eqref{eq:liminf_alpha01} by \eqref{eq:limsup_alpha1} and \eqref{eq:liminf_alpha1}, respectively.
\end{proof}

\begin{proof}[Proof of Theorem~\ref{thm:QHK}(iii)] Let $\alpha>1$. Given the scaling limit of Theorem~\ref{thm:SL}(iii), together with the fact that, $\prob$-a.s.,
\[n^{-1}\sum_{z}\tau_z\delta_{n^{-1}z}\]
converges vaguely to the Lebesgue measure multiplied by $\mathbf{E}(\tau_0)$, it is straightforward to check that a continuous-time version of \cite[Assumption 1]{CHLLT} with scaling factors $\alpha(n)=n, \beta(n)=n\mathbf{E}(\tau_0)$ and $\gamma(n)=n^2\mathbf{E}(\tau_0)$ holds when we consider the graphs $G^n$ of \cite{CHLLT} to be given by $n^{-1}\mathbb{Z}$ (equipped with nearest neighbour edges), $d_{G^n}$ to be the usual shortest path metric on $G^n$, $\nu^n$ to be the measure on $n^{-1}\mathbb{Z}$ placing mass $\tau_x$ at $x$, and $X^n$ to be the continuous-time random walk of the associated symmetric Bouchaud trap model. Thus, as commented at the end of the introduction of \cite{CHLLT}, it is possible to deduce Theorem~\ref{thm:QHK}(iii) from a continuous-time version of \cite[Theorem~1]{CHLLT} if one can further establish a continuous-time version of \cite[Assumption 2]{CHLLT}, i.e., $\prob$-a.s., for every $x_0\in[0,\infty)$ and compact interval $I\subseteq(0,\infty)$,
\begin{equation}\label{ec}
\lim_{\delta\rightarrow0}\limsup_{\lambda\rightarrow\infty}\sup_{\substack{x,y:\:|x|,|y|\leq x_0\\|x-y|\leq \delta}}\sup_{t\in I}
\mathbf{E}(\tau_0)\lambda
\left| \, p^\tau_{\mathbf{E}(\tau_0)\lambda^2t}\left(0,\lfloor \lambda x\rfloor\right)-\, p^\tau_{\mathbf{E}(\tau_0)\lambda^2t}\left(0,\lfloor \lambda y\rfloor\right)\right|=0.
\end{equation}
Towards checking this, we first observe that, $\mathbf{P}$-a.s., $\tau_x\geq 1$ for all $x$ and so $V(0,n)\geq
n$. Therefore we have from Proposition~\ref{prop:hke}(i) that, $\mathbf{P}$-a.s.\
\begin{equation}\label{odu}
p^\tau_t(0,0)\leq C \, t^{-1/2}.
\end{equation}
Next, it is a simple application of the Cauchy-Schwarz inequality to deduce that, for all $x,y\in\mathbb{Z}$,
\[\left|p^\tau_t(0,x)-p^\tau_t(0,y)\right|^2\leq |y-x|\sum_z\left(p^\tau_t(0,z)-p^\tau_t(0,z+1)\right)^2\]
The sum in the right-hand side here is the Dirichlet energy of the function $p^\tau_t(0,\cdot)$, and the argument of \cite[Proposition~4.16]{Bar} gives that this is bounded above by $t^{-1}p^\tau_t(0,0)$. Consequently, recalling the bound at \eqref{odu}, we find that, $\prob$-a.s.,
\begin{equation}\label{eq:holder-HK}
\left|p^\tau_t(0,x)-p^\tau_t(0,y)\right|^2\leq C|y-x|t^{-3/2},\qquad \forall x,y\in\mathbb{Z}.\end{equation}
From this estimate, the equicontinuity condition at \eqref{ec} readily follows, and thus the proof is complete.
\end{proof}

\begin{proof}[Proof of Theorem~\ref{thm:AHK}(i)] The result can be checked by applying a continuous-time version of \cite[Proposition~1.3]{KM}. The required input is given by the limit
\begin{equation}\label{vconv}
\lim_{\lambda\rightarrow\infty}\inf_{n\in\mathbb{N}}\prob\left(\lambda^{-1}v_\alpha(n)\leq V(0,n)\leq \lambda v_\alpha(n)\right)= 1,
\end{equation}
where
\begin{equation}\label{vdef}
v_\alpha(r):=\left\{
                 \begin{array}{ll}
                   r^{1/\alpha}, & \hbox{if $\alpha\in (0,1)$;} \\
                   r\max\{1,\log r\}, & \hbox{if $\alpha=1$;} \\
                   r, & \hbox{if $\alpha>1$.}
                 \end{array}
               \right.
\end{equation}
Since \eqref{vconv} is a straightforward consequence of the convergence in distribution of $V(0,n)/v_\alpha(n)$ to a random variable $S_\alpha$ that takes values in $(0,\infty)$ (cf.\ the convergence results described in the proof of Lemma~\ref{lem:median}), we are done.
\end{proof}

\begin{proof}[Proof of Theorem~\ref{thm:AHK}(ii)] For any $\varepsilon>0$, the lower bound is an immediate consequence of Theorem~\ref{thm:AHK}(i). As for the upper bound, we will require a quantitative version of \eqref{vconv}. In particular, writing $h_\alpha(r)=rv_\alpha(r)$, suppose that the following conditions hold for some $n\geq 2$, $t>0$ and $\lambda\geq 1$:
\[h^{-1}_{\alpha}(t/\lambda)\in[n,n+1),\]
\[\lambda^{-1}v_\alpha(n)\leq V(0,n-1)\leq 2V(0,n)\leq \lambda v_\alpha(n).\]
We then have from Proposition~\ref{prop:hke}(ii) that
\[p_t^\tau(0,0)\leq p_{2nV(0,n)}^\tau(0,0)\leq \frac{2}{V(0,n-1)}\leq \frac{2\lambda}{v_\alpha(n)} =\frac{2\lambda n}{h_\alpha(n)}\leq \frac{c\lambda^2n}{t}\leq \frac{c\lambda^2h_\alpha^{-1}(t/\lambda)}{t},\]
where $c$ is deterministic. So,
\[p_t^\tau(0,0)\leq \left\{    \begin{array}{ll}
                                         c\lambda^{\frac{1+2\alpha}{1+\alpha}}\phi_\alpha(t), & \hbox{if $\alpha\in (0,1)$;} \\
                                         c\lambda^{2}\phi_\alpha(t), & \hbox{if $\alpha=1$;}\\
                                         c\lambda^{\frac{3}{2}}\phi_\alpha(t), & \hbox{if $\alpha>1$,}
                                       \end{array}
                                     \right.\]
where for $\alpha=1$, we simply applied the bound $h_\alpha^{-1}(t/\lambda)\leq h_\alpha^{-1}(t)$, and otherwise, we applied the exact expression for $h_\alpha^{-1}$. Hence, setting $\theta=\frac{1+2\alpha}{1+\alpha}$ if $\alpha\in (0,1)$, $\theta=2$ if $\alpha=1$, and $\theta=\frac32$ if $\alpha>1$, and reparametrising, we can conclude that: if $h^{-1}_{\alpha}(t/\lambda^{1/\theta})\in[n,n+1)$ for some $n\geq 2$, then
\begin{align}
\lefteqn{\prob\left(p_t^\tau(0,0)\geq c \lambda \phi_\alpha(t)\right)}\nonumber\\
&\leq \prob\left( V(0,n-1)<\lambda^{-1/\theta}v_\alpha(n)\right)+\prob\left(2V(0,n)\geq \lambda^{1/\theta} v_\alpha(n)\right)\nonumber\\
&\leq C\lambda^{-c'/\theta},\label{ub1}
\end{align}
where we have applied Proposition~\ref{prop:tail} to obtain the second inequality. Furthermore, if it holds that $h^{-1}_{\alpha}(t/\lambda^{1/\theta})<2$, then $t\leq \lambda^{1/\theta}h_\alpha(2)$, which implies
\[\lambda \phi_\alpha(t)\geq \lambda \phi_\alpha(c\lambda^{1/\theta})\geq \lambda^\varepsilon\]
for suitably small $\varepsilon$. Consequently, applying the basic facts that $p_t^\tau(0,0)\leq \tau_0^{-1}$ and $\tau_x\geq 1$, as well as Markov's inequality,
\[\prob\left(p_t^\tau(0,0)\geq c \lambda \phi_\alpha(t)\right)\leq \prob\left(\tau_0^{-1}\geq c \lambda^\varepsilon\right)
\leq \frac{\mean[\tau_0^{-p}]}{(c \lambda^\varepsilon)^p}\leq C\lambda^{-c'},\]
where $c'$ can be taken arbitrarily large by choosing $p>0$ large. 
Hence, we have shown that: for all $t>0$ and $\lambda\geq 1$,
\[\prob\left(p_t^\tau(0,0)\geq c \lambda \phi_\alpha(t)\right)\leq 
C\lambda^{-c'}.\]
This readily implies the upper bound of Theorem~\ref{thm:AHK}(ii). To check the final claim of Theorem~\ref{thm:AHK}, it will suffice to check that 
$c'$ can 
be taken to be strictly greater than 1 in the previous inequality when $\alpha>3/2$. This is straightforward since, when $\alpha>3/2$, we have $\theta=3/2$, and Proposition~\ref{prop:tail} allows us to take 
$c'>3/2$ in \eqref{ub1}.
\end{proof}

\section{Quantitative homogenization for $\alpha>2$}\label{sec:qh}

\subsection{Quantitative CLT}

In this subsection, we will prove Theorem~\ref{thm:be}. Our argument will be based on the following general quantitative central limit theorem for martingales.

\begin{theorem} \label{thm:be_hb}
Let $(N_t)_{t\geq 0}$ be a locally square-integrable martingale (with respect to some probability measure $P$) and denote by $\Delta N_t:=N_t-N_{t-}$ its jump process and by $\langle N \rangle_t$ its quadratic variation process. Then, for any $n>1$, there exists a constant $c>0$ depending only on $n$ such that
\[ \sup_{x \in \bbR} \left|  P\big[N_1 \leq  x  \ \big] - \Phi(x)\right| \;\leq \; c \, \Big( E\big[ \big| \langle N\rangle_1 -1\big|^n \big]+ E \Big[ \sum_{0\leq t \leq 1} \big| \Delta N_t\big|^{2n} \Big]\Big)^{1/(2n+1)}.\]
\end{theorem}
\begin{proof}
 See \cite[Theorem~2]{Ha88} (cf.\ \cite{HB70}).
\end{proof}

To apply this, we will establish the following quenched and annealed estimates for the BTM.

\begin{prop} \label{prop:estV}
Suppose that  $\alpha>2$ and let $\sigma^2\ldef 1/ \mean[\tau_0]$.
\begin{enumerate}
\item[(i)] It holds that
\begin{align*}
\mean \Bigg[ E_0^\tau \bigg[\Big| \frac{\langle X \rangle_t} t -\sigma^2 \Big|^2 \bigg] \Bigg] =O\left(t^{-1/2}\log t\right).
\end{align*}

\item[(ii)] For any $\eta>0$ and $\prob$-a.e.\ $\tau$, it holds that
\begin{align*}
E_0^\tau \bigg[\Big| \frac{\langle X \rangle_t} t -\sigma^2 \Big|^2 \bigg] =O\left( t^{-\frac 14+\frac{3}{4(2\alpha-1)}+\eta}\right).
\end{align*}
\end{enumerate}
\end{prop}

In order to show Proposition~\ref{prop:estV}, we consider a continuous time simple random walk $(Y_t)_{t\geq 0}$ on $\bbZ$ with unit jump rates and associated heat kernel denoted by $p^Y_t(x,y)$, and set
\begin{align*}
A_t\ldef \int_0^t  \tau_{Y_s} \, ds, \qquad t\geq 0.
\end{align*}
The latter process is a continuous time version of a random walk in random scenery, as studied in \cite{DF19, DF20}. (The original model in discrete time was first introduced by Borodin \cite{Bo79, Bo79a} and Kesten-Spitzer \cite{KS79} independently.) For any configuration $\tau=(\tau_x)_{x\in\bbZ}$, in a slight abuse of notation, we denote by $E_x^\tau$ the quenched law of $(\tau_{Y_t})_{t\geq 0}$ when the random walk $(Y_t)_{t\geq 0}$ starts at $x\in \bbZ$. Importantly, we note that  \[X_t=Y_{A_t^{-1}},\]
where $A^{-1}$ denotes the right-continuous inverse of $A$. We will show that Proposition~\ref{prop:estV} is a consequence of the following statement on the time-change functional $A$.

\begin{prop}\label{prop4-3}
Suppose that  $\alpha>2$.
\begin{enumerate}
\item[(i)] It holds that
\begin{align} \label{eq:At_annealed}
\mean\bigg[ E_0^\tau \Big[\big| A_t - \mean[\tau_0] t \big|^2 \Big] \bigg] =O\left( t^{3/2}\right).
\end{align}

\item[(ii)] For any $\eta>0$ and $\prob$-a.e.\ $\tau$, it holds that
\[E_0^\tau \Big[\big| A_t - \mean[\tau_0] t \big|^2 \Big]  =O\left(t^{\frac74+\frac{3}{4(2\alpha-1)}+\eta}\right).\]
\end{enumerate}
\end{prop}

\begin{proof}
(i) For any $t\geq 0$, by the on-diagonal heat kernel bound $p^Y_t(x,x)\leq c(1\wedge t^{-1/2})$ of \cite[Theorem~6.28]{Barbook},
\begin{align*}
\mean\bigg[ E_0^\tau \Big[\big| A_t - \mean[\tau_0] t \big|^2 \Big] \bigg]&=\int_0^t \int_0^t \Big( \mean[ E^\tau_0[\tau_{Y_u} \tau_{Y_v}]  ] - \mean[\tau_0]^2 \Big) \, du \, dv \\
&=2 \, \mathrm{Var}(\tau_0)  \int_0^t \int_v^t  \, p^Y_{u-v}(0,0) \, du \, dv\\
&\leq c \int_0^t \int_v^t  \, \big(1\wedge (u-v)^{-1/2} \big)  \, du \, dv \\
&\leq c  \int_0^t  \big(1+ (t-v)^{1/2} \big) \, dv \\
&\leq c \, t^{3/2},
\end{align*}
which shows (i).
\medskip

(ii) For this part of the proof, we will use similar arguments to those of \cite{DF19}.  Define $\bar A_t\ldef \int_0^t  \bar\tau_{Y_s} \, ds$, $t\geq 0$, where $\bar\tau_x \ldef \tau_x - \mean[\tau_0]$, $x\in \bbZ$. Further, for any $\varepsilon>0$, set $\bar A_t^\varepsilon \ldef \int_0^t  \bar\tau_{Y_s}\wedge t^\varepsilon \, ds$, $t\geq 0$, so that we have the obvious bound $|\bar A_t^\varepsilon| \leq t^{1+\varepsilon}$.

\smallskip
\noindent
\emph{Step~1.} In our first step, we show that for $\prob$-a.e.\ $\tau$,
\begin{align} \label{eq:bound_Aeps}
\Big| E_0^\tau \big[\big| \bar A_t^\varepsilon \big|^2 \big] -\mean\big[ E_0^\tau \big[\big| \bar A_t^\varepsilon \big|^2 \big] \big]\Big|=O\left( t^{7/4+3\varepsilon}\right).
\end{align}
To this end, we start by noting that
\[E_0^\tau \big[\big| \bar A_t^\varepsilon \big|^2 \indicator_{\{\sup_{s\leq t}|Y_s|\geq 2t\}}\big]\leq 
t^{2(1+\varepsilon)}P_0\left(\sup_{s\leq t}|Y_s|\geq 2t\right)\leq Ce^{-ct},\]
where the exponential bound on the probability $P_0(\sup_{s\leq t}|Y_s|\geq 2t)$ follows from \cite[Lemmas~5.21, 5.22 and Theorem~6.28]{Barbook}. Hence it will suffice to check \eqref{eq:bound_Aeps} with $E_0^\tau [| \bar A_t^\varepsilon |^2 ]$ replaced by  $E_0^\tau [| \bar A_t^\varepsilon |^2 \indicator_{\sup_{s\leq t}|Y_s|< 2t}]$. For two configurations $\tau^1$ and $\tau^2$, by using the Cauchy-Schwarz inequality,
\begin{align*}
\lefteqn{E_0^{\tau^1} \big[\big| \bar A_t^\varepsilon \big|^2 \indicator_{\{\sup_{s\leq t}Y_s< t\}} \big] -E_0^{\tau^2} \big[\big| \bar A_t^\varepsilon \big|^2 \indicator_{\{\sup_{s\leq t}|Y_s|< 2t\}}\big]}\\
&=E_0^{\tau^1, \tau^2} \bigg[\int_0^t \big( \bar\tau^1_{Y_s}\wedge t^\varepsilon - \bar\tau^2_{Y_s}\wedge t^\varepsilon \big)  \, ds \, \big(  \bar A_t^\varepsilon(\tau^1) +  \bar A_t^\varepsilon(\tau^2) \big)\indicator_{\{\sup_{s\leq t}|Y_s|<2t\}}\bigg] \\
& \le 2 t^{1+\varepsilon} E_0^{\tau^1, \tau^2} \bigg[\int_0^t \big| \bar\tau^1_{Y_s}\wedge t^\varepsilon - \bar\tau^2_{Y_s}\wedge t^\varepsilon \big| \indicator_{\{|Y_s|< 2t\}} \, ds \,\bigg]\\
& \leq 2 t^{1+\varepsilon} \int_0^t \sum_{y\in [-2t,2t]} p^Y_s(0,y) \big| \bar\tau^1_{y}\wedge t^\varepsilon - \bar\tau^2_{y}\wedge t^\varepsilon \big|  \, ds\\
& \leq 2 t^{1+\varepsilon}  \int_0^t \sum_{y\in [-2t,2t]} p^Y_s(0,y) \big| \bar\tau^1_{y}\wedge t^\varepsilon - \bar\tau^2_{y}\wedge t^\varepsilon \big|  \, ds \\
& \leq 2 t^{1+\varepsilon}  \int_0^t  \bigg(\sum_{y\in [-2t,2t]} p^Y_s(0,y)^2 \bigg)^{1/2}  \bigg(\sum_{y\in [-2t,2t]} \big| \bar\tau^1_{y}\wedge t^\varepsilon - \bar\tau^2_{y}\wedge t^\varepsilon \big|^2 \bigg)^{1/2} \, ds  \\
& \leq 2 t^{1+\varepsilon} \, \| \bar\tau^1\wedge t^\varepsilon - \bar\tau^2\wedge t^\varepsilon\|_2 \int_0^t p^Y_{2s}(0,0)^{1/2} \, ds\\
& \leq 2 t^{7/4+\varepsilon} \, \| \bar\tau^1\wedge t^\varepsilon - \bar\tau^2\wedge t^\varepsilon\|_2,
\end{align*}
where we write $E_0^{\tau^1,\tau^2}\ldef E_0^{\tau^1} \times  E_0^{\tau^2}$ and   $\|\tau\|_2:=(\sum_{y\in[-2t,2t]} \tau_y^2)^{1/2}$. This allows us to apply Talagrand's concentration inequality (see  \cite[Theorem~6.6]{Ta96} and the argument below \cite[Equation~(4.2)]{Ta96} to replace the median by the mean) to obtain
\begin{align*}
\lefteqn{\prob\bigg[ \Big| E_0^\tau \big[\big| \bar A_t^\varepsilon \big|^2\indicator_{\sup_{s\leq t}|Y_s|< 2t} \big] -\mean\big[ E_0^\tau \big[\big| \bar A_t^\varepsilon \big|^2 \indicator_{\sup_{s\leq t}|Y_s|< 2t}\big] \big]\Big|> t^{7/4+3\varepsilon}  \bigg]}\\
&  \; \leq \;  4 \exp\big( - t^{7/2+6\varepsilon}/ (16 t^{7/2+4\varepsilon})\big) \hspace{160pt}\\
  & \; \leq\; C \exp\big(-c t^{2\varepsilon}\big).
\end{align*}
(We use \cite[Theorem~6.6]{Ta96} with the choice $f(\tau)= E_0^\tau \big[\big| \bar A_t^\varepsilon \big|^2 \indicator_{\sup_{s\leq t}|Y_s|< 2t} \big]$, $\sigma=t^{7/4+2\varepsilon}$, $X_x=t^{-\varepsilon}(\tau_x\wedge t^\varepsilon) \in [0,1]$, $N=|[-2t,2t] \cap \bbZ|$ and $b=0$.)
Hence, by the Borel-Cantelli lemma we obtain \eqref{eq:bound_Aeps}, as required.

\smallskip
\noindent
\emph{Step~2.} We next show that, for an arbitrary $\delta>0$ and $\prob$-a.e.\ $\tau$,
\begin{equation}\label{step2}
E_0^\tau\left[ \left(\bar A_t-\bar A_t^\varepsilon  \right)^{\!2}\right]\leq E_0^\tau\left[ \left( \int_0^t \bar\tau_{Y_s} \indicator_{\{ \bar\tau_{Y_s} > t^\varepsilon\}}  \right)^{\!2}\right]=O\left(t^{2+2\delta-2(\alpha-2)\varepsilon}\right).
\end{equation}
(The first inequality is obvious, and so it will suffice to derive a bound for the central expression.) To this end, we first note that if $(p_t^Y(x,y))_{x,y\in\mathbb{Z},t>0}$ is the transition density of $Y$, then the standard Gaussian bounds that hold for this (see \cite[Theorem~6.28]{Barbook}) imply that
\begin{align*}
\int_{0}^tp^Y_s(x,y)ds&\leq \int_0^tCs^{-1/2}e^{-\frac{|x-y|^2}{Cs}\wedge \frac{|x-y|}{C}}ds\\
&\leq Ce^{-\frac{|x-y|^2}{Ct}\wedge \frac{|x-y|}{C}}\int_0^ts^{-1/2}ds\\
&=2Ct^{1/2}e^{-\frac{|x-y|^2}{Ct}\wedge \frac{|x-y|}{C}}.
\end{align*}
Hence
\begin{eqnarray}
\lefteqn{E_0^\tau\left[ \left( \int_0^t \bar\tau_{Y_s} \indicator_{\{ \bar\tau_{Y_s} > t^\varepsilon\}}  \right)^{\!2}  \right]}\nonumber\\
&=&2\sum_{x,y\in\mathbb{Z}}
 \bar\tau_{x} \indicator_{\{ \bar\tau_{x} > t^\varepsilon\}}  \bar\tau_{y} \indicator_{\{ \bar\tau_{y} > t^\varepsilon\}}\int_0^t\int_0^sp^Y_r(0,y) \, p^Y_{s-r}(y,x) \,dr \, ds\nonumber\\
 &\leq&Ct\sum_{x,y\in\mathbb{Z}}
 \bar\tau_{x} \indicator_{\{ \bar\tau_{x} > t^\varepsilon\}}  \bar\tau_{y} \indicator_{\{ \bar\tau_{y} > t^\varepsilon\}}e^{-\frac{|x|^2}{Ct}\wedge \frac{|x|}{C}}e^{-\frac{|x-y|^2}{Ct}\wedge \frac{|x-y|}{C}}.\label{return}
\end{eqnarray}
To bound the sum here, we will use the following basic almost-sure bounds on the i.i.d.\ random variables in the collection $(\tau_x)_{x\in\mathbb{Z}}$. First, for each $\delta>0$, it $\mathbf{P}$-a.s.\ holds that there exists a (possibly random) constant $C$ such that
\begin{equation}\label{tausup}
\sup_{|x|\leq n}\tau_x\leq Cn^{\frac{1}{\alpha}+\delta},\qquad \forall n\in \mathbb{N}.
\end{equation}
(To check this, we simply note that 
\[\mathbf{P}\left(\sup_{|x|\leq 2^n}\tau_x>2^{n(\frac{1}{\alpha}+\delta)}\right)\leq c 2^n \mathbf{P}\left(\tau_0>2^{n(\frac{1}{\alpha}+\delta)}\right)=c2^{-n\alpha\delta},\]
from which the bound at \eqref{tausup} follows from a simple Borel-Cantelli argument.) Second, writing $N(n,m):=\#\{x:\:|x|\leq n,\:\bar\tau_{x} \geq m\}$, if $\beta>0$, then it $\mathbf{P}$-a.s.\ holds that there exists a (possibly random) constant $C$ such that
\begin{equation}\label{Nbound}
  N(n,n^\beta)\leq Cn^{1-\alpha\beta},\qquad \forall n\in \mathbb{N}.
\end{equation}
(Similarly, to check this, we observe 
\begin{align*}
\lefteqn{\mathbf{P}\left(N(2^{n+1},2^{n\beta})\geq 4\times 2^{n(1-\alpha\beta)}\right)}\\
&\leq 
\mathbf{P}\left(|N(2^{n+1},2^{n\beta})-\mathbf{E}N(2^{n+1},2^{n\beta})|\geq 2\times 2^{n(1-\alpha\beta)}\right)\\
&\leq c 2^{-2n(1-\alpha\beta)}\mathrm{Var}\left(N(2^{n+1},2^{n\beta})\right)\\
&\leq c2^{-n(1-\alpha\beta)},
\end{align*}
which again is summable, and so gives the bound at \eqref{Nbound} by a Borel-Cantelli argument.) To apply these bounds, it will be convenient to break the sum $\sum_{x,y\in\mathbb{Z}}$ at \eqref{return} into four pieces:
\[\sum_{\substack{x\in\mathbb{Z}:\\|x|\leq t^{\frac12+\delta}}}\sum_{\substack{y\in\mathbb{Z}:\\|y-x|\leq t^{\frac12+\delta}}}+
 \sum_{\substack{x\in\mathbb{Z}:\\|x|> t^{\frac12+\delta}}}\sum_{\substack{y\in\mathbb{Z}:\\|y-x|\leq t^{\frac12+\delta}}}+
 \sum_{\substack{x\in\mathbb{Z}:\\|x|\leq t^{\frac12+\delta}}}\sum_{\substack{y\in\mathbb{Z}:\\|y-x|> t^{\frac12+\delta}}}+
 \sum_{\substack{x\in\mathbb{Z}:\\|x|> t^{\frac12+\delta}}}\sum_{\substack{y\in\mathbb{Z}:\\|y-x|> t^{\frac12+\delta}}};\]
 we will call these pieces $\Sigma_1$, $\Sigma_2$, $\Sigma_3$, $\Sigma_4$, respectively. Applying \eqref{tausup} and assuming $\delta\in (0,1/2)$, for the fourth piece we have that
 \begin{eqnarray*}
 \Sigma_4&\leq &C\sum_{\substack{x\in\mathbb{Z}:\\|x|> t^{\frac12+\delta}}}|x|^{\frac{1}{\alpha}+\delta}e^{-C^{-1}|x|^{4\delta/(1+2\delta)}}\sum_{\substack{y\in\mathbb{Z}:\\|y-x|> t^{\frac12+\delta}}} |y|^{\frac{1}{\alpha}+\delta}e^{-C^{-1}|y-x|^{4\delta/(1+2\delta)}}\\
 &\leq&C\sum_{\substack{x\in\mathbb{Z}:\\|x|> t^{\frac12+\delta}}}|x|^{\frac{1}{\alpha}+\delta}e^{-C^{-1}|x|^{4\delta/(1+2\delta)}}\sum_{\substack{y\in\mathbb{Z}:\\|y|> t^{\frac12+\delta}}} \left(|y|^{\frac{1}{\alpha}+\delta}+|x|^{\frac{1}{\alpha}+\delta}\right)e^{-C^{-1}|y|^{4\delta/(1+2\delta)}}\\
 &\leq &C\left(\sum_{\substack{x\in\mathbb{Z}:\\|x|> t^{\frac12+\delta}}}|x|^{\frac{2}{\alpha}+2\delta}e^{-C^{-1}|x|^{4\delta/(1+2\delta)}}\right)^2,
 \end{eqnarray*}
and it is elementary to check that this decays to 0 faster than any polynomial as $t\rightarrow\infty$. For the third piece, we proceed similarly, applying \eqref{tausup} to deduce that
 \begin{eqnarray*}
 \Sigma_3&\leq &C\sum_{\substack{x\in\mathbb{Z}:\\|x|\leq t^{\frac12+\delta}}}|x|^{\frac{1}{\alpha}+\delta}e^{-C^{-1}|x|^{4\delta/(1+2\delta)}}\sum_{\substack{y\in\mathbb{Z}:\\|y-x|> t^{\frac12+\delta}}} |y|^{\frac{1}{\alpha}+\delta}e^{-C^{-1}|y-x|^{4\delta/(1+2\delta)}}\\
 &\leq &Ct^{\left(\frac12+\delta\right)\left(\frac1\alpha+\delta\right)}\sum_{\substack{y\in\mathbb{Z}:\\|y|> t^{\frac12+\delta}}} |y|^{\frac{1}{\alpha}+\delta}e^{-C^{-1}|y|^{4\delta/(1+2\delta)}}\\
&\leq &C\sum_{\substack{y\in\mathbb{Z}:\\|y|> t^{\frac12+\delta}}} |y|^{\frac{2}{\alpha}+2\delta}e^{-C^{-1}|y|^{4\delta/(1+2\delta)}},
\end{eqnarray*}
and again one can check that this expression converges to 0 faster than any polynomial as $t\rightarrow\infty$. Also, for the second piece, one has from a similar argument that
   \[\Sigma_2\leq C\sum_{\substack{x\in\mathbb{Z}:\\|x|> t^{\frac12+\delta}}} |x|^{\frac{2}{\alpha}+2\delta}e^{-C^{-1}|x|^{4\delta/(1+2\delta)}},\]
   which also converges to 0 faster than any polynomial as $t\rightarrow\infty$. Thus it remains to bound the first term, for which we straightforwardly have that
   \[  \Sigma_1\leq  \sum_{\substack{x\in\mathbb{Z}:\\|x|\leq t^{\frac12+\delta}}}\sum_{\substack{y\in\mathbb{Z}:\\|y-x|\leq t^{\frac12+\delta}}}\bar\tau_{x} \indicator_{\{ \bar\tau_{x} > t^\varepsilon\}}  \bar\tau_{y} \indicator_{\{ \bar\tau_{y} > t^\varepsilon\}}\leq 
  \left(\sum_{\substack{x\in\mathbb{Z}:\\|x|\leq 2t^{\frac12+\delta}}}
   \bar\tau_{x} \indicator_{\{ \bar\tau_{x} > t^\varepsilon\}}\right)^2.\]
   Now, it readily follows that
\[\sum_{\substack{x\in\mathbb{Z}:\\|x|\leq 2t^{\frac12+\delta}}}
\bar\tau_{x} \indicator_{\{ \bar\tau_{x} > t^\varepsilon\}}\leq \sum_{k=1}^\infty 
N\left( 2t^{\frac12+\delta},t^{\varepsilon k}\right)t^{(k+1)\varepsilon}.\]
Moreover, from \eqref{tausup}, when $t$ is sufficiently large, the terms in this expression are equal to 0 for $k\geq  k_0$, where $k_0:= 1+\lceil\varepsilon^{-1}(\frac12+\delta)(\frac1\alpha+\delta)\rceil$. Therefore, recalling the bound at \eqref{Nbound}, we have that
\[\sum_{\substack{x\in\mathbb{Z}:\\|x|\leq 2t^{\frac12+\delta}}}\bar\tau_{x} 
\indicator_{\{ \bar\tau_{x} > t^\varepsilon\}}\leq C\sum_{k=1}^{k_0}t^{\frac12+\delta-\alpha\varepsilon k}t^{(k+1)\varepsilon}=Ct^{\frac12+\delta+\varepsilon}\sum_{k=1}^{k_0}t^{-(\alpha-1)k\varepsilon}.\]
Since $\alpha\geq 1$, the sum is bounded by $k_0t^{-(\alpha-1)\varepsilon}$, and so we obtain (for sufficiently large $t$)
\[\sum_{\substack{x\in\mathbb{Z}:\\|x|\leq 2t^{\frac12+\delta}}}
\bar\tau_{x} \indicator_{\{ \bar\tau_{x} > t^\varepsilon\}}\leq Ct^{\frac12+\delta-(\alpha-2)\varepsilon}.\]
This implies that $\Sigma_1=O(t^{1+2\delta-2(\alpha-2)\varepsilon})$, and combining this with the decay of the other terms, we deduce from \eqref{return} that \eqref{step2} holds.

\smallskip
\noindent
\emph{Step~3.} This is the annealed version of Step 2. In particular, we aim to check that, for $\prob$-a.e.\ $\tau$,
\begin{equation}\label{step3}
\mathbf{E}\left[E_0^\tau\left[ \left( \int_0^t \bar\tau_{Y_s} \indicator_{\{ \bar\tau_{Y_s} > t^\varepsilon\}}  \right)^{\!2}\right]\right]=O\left(t^{\frac32-(\alpha-2)\varepsilon}\right)+O\left(t^{2-2(\alpha-1)\varepsilon}\right).
\end{equation}
To check this, we first take expectations at \eqref{return} to yield
\begin{align}
\lefteqn{\mathbf{E}\left[E_0^\tau\left[ \left( \int_0^t \bar\tau_{Y_s} \indicator_{\{ \bar\tau_{Y_s} > t^\varepsilon\}}  \right)^{\!2}  \right]\right]}\nonumber\\
&\leq Ct\sum_{x,y\in\mathbb{Z}}
 \mathbf{E}\left[\bar\tau_{x} \indicator_{\{ \bar\tau_{x} > t^\varepsilon\}}  \bar\tau_{y} \indicator_{\{ \bar\tau_{y} > t^\varepsilon\}}\right]e^{-\frac{|x|^2}{Ct}\wedge \frac{|x|}{C}}e^{-\frac{|x-y|^2}{Ct}\wedge \frac{|x-y|}{C}}.\label{step32}
 \end{align}
Now, an elementary calculation gives
\[\mathbf{E}\left[\bar\tau_{x} \indicator_{\{ \bar\tau_{x} > t^\varepsilon\}}  \bar\tau_{y} \indicator_{\{ \bar\tau_{y} > t^\varepsilon\}}\right]\leq \left\{
                                             \begin{array}{ll}
                                               Ct^{-(\alpha-2)\varepsilon}, & \hbox{if $x=y$;} \\
                                               Ct^{-2(\alpha-1)\varepsilon}, & \hbox{otherwise.}
                                             \end{array}
                                           \right.\]
Using these estimates and the fact that
\[\sum_{x\in\mathbb{Z}}e^{-\frac{|x|^2}{Ct}\wedge \frac{|x|}{C}}\leq Ct^{1/2},\]
we obtain \eqref{step3} from \eqref{step32}, where the two terms correspond to the sums over $x=y$ and that over $x\neq y$, respectively.

\smallskip
\noindent
\emph{Step~4.} Finally, putting together the estimates of \eqref{eq:At_annealed}, \eqref{eq:bound_Aeps}, \eqref{step2} and \eqref{step3}, we find that, for any $\delta>0$ and $\prob$-a.e.\ $\tau$, it holds that
\[E_0^\tau \Big[\big| A_t - \mean[\tau_0] t \big|^2 \Big]  =O\left(t^{7/4+3\varepsilon}\right)+O\left(t^{2+2\delta-2(\alpha-2)\varepsilon}\right).\]
As $\delta$ can be chosen arbitrarily small and $\alpha>2$, optimising over $\varepsilon$ gives an upper bound of the desired form, i.e.\
\[O\left(t^{\frac74+\frac{3}{4(2\alpha-1)}+\eta}\right),\]
for arbitrarily small $\eta$.
\end{proof}

\begin{remark}
Whilst the exponent improves from 2 to $7/4$ as $\alpha$ is increased from 2 to $\infty$, we do not obtain the bound of $t^{3/2}$ that holds in the deterministic case (i.e.\ when the $\tau_x$ are identically equal to some deterministic constant). We expect that we give something away in the exponent due to our use of Cauchy-Schwarz in Step~1 of the proof of part (ii) of the above result. It would be of interest to determine the optimal exponent for the Bouchaud trap model, and how close the bound is to $t^{3/2}$ for bounded random variables.
\end{remark}

\begin{proof}[Proof of Proposition~\ref{prop:estV}]
Set $\sigma^2 \ldef 1/ \mean[\tau_0]$. First notice that, for $\prob$-a.e. $\tau$, $A^{-1}_t \leq t$ for all $t\geq 0$. In particular,
\[P_0^\tau \Big[ \big| A_t^{-1} - \sigma^2 t \big| \geq x \Big]=0\]
for all $x\geq (1-\sigma^2)t$. Moreover, by Proposition~\ref{prop4-3}(ii) we have for $x\in (0, \sigma^2t]$,
\begin{align*}
\lefteqn{P_0^\tau \Big[ \big| A_t^{-1} - \sigma^2 t \big| \geq x \Big]}\\
&=P_0^\tau \Big[  A_t^{-1} - \sigma^2 t  \geq x \Big] +P_0^\tau \Big[  A_t^{-1} - \sigma^2 t  \leq -x \Big] \\
&=P_0^\tau \Big[  A_{\sigma^2t +x} \leq t  \Big]+P_0^\tau \Big[  A_{\sigma^2t -x} \geq t  \Big] \\
&\leq P_0^\tau \Big[ \big| A_{\sigma^2 t +x} - \sigma^{-2} (\sigma^2 t +x) \big| \geq x/\sigma^2   \Big]+P_0^\tau \Big[ \big| A_{\sigma^2t -x} - \sigma^{-2} (\sigma^2 t -x) \big| \geq x/\sigma^2   \Big]\\
& \leq\frac{\sigma^4}{x^2} \bigg( E_0^\tau \Big[  \big| A_{\sigma^2t +x} - \sigma^{-2} (\sigma^2 t +x) \big|^2  \Big]+E_0^\tau \Big[  \big| A_{\sigma^2t -x} - \sigma^{-2} (\sigma^2 t -x) \big|^2  \Big]
\bigg)\\
& \leq\frac{c \sigma^4}{x^2}  \big( \sigma^2t+x \big)^{\frac74+\frac{3}{4(2\alpha-1)}+\eta},
\end{align*}
and similarly, for $x\geq \sigma^2t$,
\[P_0^\tau \Big[ \big| A_t^{-1} - \sigma^2 t \big| \geq x \Big]=P_0^\tau \Big[  A_{\sigma^2t+x} \leq t  \Big]  \leq\frac{c \sigma^4}{x^2}  \big( \sigma^2t+x \big)^{\frac74+\frac{3}{4(2\alpha-1)}+\eta}.\]
Hence,
\begin{align} \label{eq:L2_Ainv}
E_0^\tau \Big[\big| A^{-1}_t - \sigma^2 t \big|^2 \Big]
&= \int_0^{\infty}P_0^\tau \Big[ \big| A_t^{-1} - \sigma^2 t \big|^2 \geq x \Big] \, dx \nonumber \\
& \leq\int_0^{(1-\sigma^2)^2t^2}P_0^\tau \Big[ \big| A_t^{-1} - \sigma^2 t \big| \geq \sqrt{x} \Big] \, dx \nonumber \\
& \leq\int_0^{(1-\sigma^2)^2t^2} \bigg( 1 \wedge \frac{c \sigma^4}{x}  \big( \sigma^2t+\sqrt{x} \big)^{\frac74+\frac{3}{4(2\alpha-1)}+\eta}  \bigg) \, dx   \nonumber \\
& =O\left( t^{\frac74+\frac{3}{4(2\alpha-1)}+\eta} \log t\right).
\end{align}
In order to show the statement, recall that $\langle X\rangle_t = \langle Y \rangle_{A^{-1}_t}$, where $\langle Y \rangle_t$ equals the number of jumps of $Y$ up to time $t$. In particular, $(\langle Y\rangle_t)_{t\geq 0}$ is a Poisson process with rate $1$ and $(\langle Y\rangle_t - t)_{t\geq 0}$ is martingale. Thus, by decomposing $\langle X \rangle_t$ as $\langle X \rangle_t=A_t^{-1}+ \langle Y \rangle_{A^{-1}_t}- A_t^{-1}$,  we get
\begin{align*}
E_0^\tau \bigg[\Big| \frac{\langle X \rangle_t} t -\sigma^2 \Big|^2 \bigg]&\leq2 \bigg( E_0^\tau \bigg[\Big| \frac{A_t^{-1}} t -\sigma^2 \Big|^2 \bigg]+ t^{-2} E_0^\tau \Big[\big( \langle Y \rangle_{A_t^{-1}} - A^{-1}_t \big)^2 \Big] \bigg) \\
& \leq O\left( t^{-\frac14+\frac{3}{4(2\alpha-1)}+\eta} \log t\right) + t^{-2} E_0^\tau \Big[\sup_{s\leq t} \big(  \langle Y \rangle_{s} - s \big)^2 \Big],
\end{align*}
where we used \eqref{eq:L2_Ainv} and again that $A^{-1}_t \leq t$. Using Doob's inequality and observing that  $E_0^\tau \big[\big(  \langle Y \rangle_{t} - t \big)^2 \big]$ equals the variance of a Poisson distributed random variable with parameter $t$, we can bound the second term from above as
\begin{align*}
t^{-2} E_0^\tau \Big[\sup_{s\leq t} \big(  \langle Y \rangle_{s} - s \big)^2 \Big]  \leq c t^{-2} E_0^\tau \Big[\big(  \langle Y \rangle_{t} - t \big)^2 \Big]  =O\left( t^{-1}\right).
\end{align*}
Since $\eta>0$ is arbitrary, this shows (ii). Part~(i) follows from Proposition~\ref{prop4-3}(i) by the same arguments.
\end{proof}

We next control the jumps of $X$.

\begin{prop} \label{prop:estJ}
For any $n\in \bbN$  and $\bbP$-a.e.\ $\tau$,
\[ E_0^\tau \Big[ \sum_{0\leq s \leq t} \big|X_s - X_{s-}\big|^n  \Big] \; \leq \; t, \qquad \forall t>0.\]
\end{prop}

\begin{proof}
According to the L\'evy system formula
(see e.g.\ \cite[Lemma 4.7]{CK03}, 
\cite[Theorem VI.28.1]{RW00} or \cite{BJ73}), for any continuous bounded non-negative function $f:\bbZ^d \times \bbZ^d \to \bbR$ that vanishes on the diagonal,  for $\bbP$-a.e.\ $\tau$,
\[
E_0^\tau \bigg[ \sum_{0\leq s\leq t} f(X_{s-},X_s) \bigg] = E_0^\tau \bigg[ \int_{(0,t]}  \frac 1 {2 \tau_{X_{s-}}}  \sum_{y \sim X_{s-}}  f(X_{s-},y) \, ds \bigg].
\]
Then choosing $ f(x,y)= 1\wedge \big| x-y \big|^n$ and recalling that $\tau_x\geq 1 $ for all $x$, we obtain
\[E_0^\tau \Big[ \frac 1 t \sum_{0\leq s\leq t} \big| X_s- X_{s-} \big|^n \Big]=
\frac 1 t \int_0^t ds \, E^\tau_0\Big[  \frac 1 {2 \tau_{X_{s-}}} \sum_{y\sim X_{s-} }
  \big|  y -   X_{s-} \big|^n \Big]  \leq 1,\]
which finishes the proof.
\end{proof}

\begin{proof} [Proof of Theorem~\ref{thm:be}]
We shall apply the general result of Theorem~\ref{thm:be_hb} with the choice $n=2$ on the martingale
\[ N_s \; \ldef \; \frac{X_{st}}{\sqrt t \, \sigma}, \qquad 0\leq s \leq 1.\]
Since
\[ \langle N \rangle_1-1 \;=\; \frac{\langle X \rangle_t}{t \, \sigma^2}-1\; =\;\frac 1 {\sigma^2} \bigg( \frac{\langle X \rangle_t}{t}- \sigma^2 \bigg),\]
we get by Proposition~\ref{prop:estV} that
\begin{align} \label{eq:estMtilde1}
\mean \Big[ E_0^\tau \big[ | \langle N \rangle_1 -1|^{2} \big] \Big]
 =O\left(t^{-1/2} \log t\right),    \:\:    E_0^\tau \big[ | \langle N \rangle_1 -1|^{2} \big] =O\left( t^{- \frac 14+\frac{3}{4(2\alpha-1)}+\eta}\right).
\end{align}
Furthermore,
\begin{align*}
 \sum_{s\leq 1} \big| \Delta N_s\big|^{4} \; = \; (t\sigma^2)^{-2} \sum_{s\leq t }  \big|X_s - X_{s-}\big|^{4},
\end{align*}
and we obtain from Proposition~\ref{prop:estJ} that
\begin{align} \label{eq:estMtilde3}
  E_0^\tau \Big[ \sum_{0\leq s \leq 1} \big| \Delta N_s\big|^{4} \Big] \; \leq t^{-1}.
\end{align}
Now we apply Theorem~\ref{thm:be_hb}, under the annealed measure $\bbP$, which gives
\begin{align*}
 & \sup_{x \in \bbR} \left|  \bbP \big[X_t \leq  x \sqrt{t} \ \big]  - \Phi(\tfrac x {\sigma})\right| \;= \;
\sup_{x \in \bbR} \left| \bbP\big[ N_1 \leq  x \ \big]  - \Phi(x)\right| \\
& \mspace{36mu} \leq \; c\Bigg( \mean \bigg[ E_0^\tau \Big[ \big| \langle N \rangle_1 -1\big|^{2} \Big] \bigg]+  \mean \bigg[ E_0^\tau \Big[ \sum_{0\leq s \leq 1} \big| \Delta N_s\big|^{4} \Big] \bigg]\Bigg)^{\frac 1 5},
\end{align*}
and an analogous estimate holds under  the quenched measure $P_0^\tau$. Hence, the claim follows from \eqref{eq:estMtilde1} and \eqref{eq:estMtilde3}.
\end{proof}

\subsection{Quantitative local CLT}

In this section, we derive the quantitative quenched local limit result of Theorem~\ref{thm:QLCLT}. We first give a general statement of a quantitative local CLT for BTM on a class of metric measure spaces and then prove it applies in our case. The approach is quite generic and will apply to other stochastic processes and settings, but we do not attempt to give the broadest presentation of the argument here, so as to not obscure the main ideas.

Let $(M,d,\mu)$ be a metric measure space, where $M$ is a locally compact metric space that is also a vector space, $d$ be a metric on $M$ and $\mu$ be a Radon measure on $M$ of full support. Let $0\in M$ be the origin of $M$ as the vector space. Let $G\subset M$ be the vertices of a locally finite, connected graph and for $n\in \mathbb Z$, let $G_n:= n^{-1}G$. Suppose $G$ consists of `cells' $\{V^{(i)}\}_{i\in \Lambda}$ such that $|V^{(i)}|=L$ for some $L\ge 2$ and $\cup_{i\in \Lambda}V^{(i)}=G$, and moreover, $M$ can be decomposed into \lq complexes' $\{K^{(i)}\}_{i\in \Lambda}$, where $K^{(i)}:=\mbox{Conv} (V^{(i)})$, with $\mbox{Conv}\,(H)$ being the convex hull of $H$. We assume that $\sup_{i\in \Lambda}\,\mbox{diam}\,K^{(i)}<\infty$, $\cup_{i\in \Lambda}K^{(i)}=M$, and there exists a $d_f\geq 1$ such that $\mu (n^{-1}K^{(i)})=n^{-d_f}\mu (K^{(i)})$ for all $i\in \Lambda$ and $n\in \mathbb Z$. In the following, we write $V^{(i)}_n:=n^{-1}V^{(i)}$ and $K^{(i)}_n:=n^{-1}K^{(i)}$. Concerning the dynamics on $M$, we assume that there exists a diffusion process $\{B_t\}_t$ on $M$ that has transition density $p_{\mathrm{BM}}(t,x,y)$ (with respect to $\mu$) enjoying the following estimates for all $t\in (0,\infty)$, $x,y\in M$,
\begin{eqnarray}
c_3t^{-d_f/d_w}\exp\Big(-c_4\Big(\frac{d(x,y)^{d_w}}t\Big)^{1/(d_w-1)}\Big)
\le p_{\mathrm{BM}}(t,x,y) \nonumber\\
~~~~~~\le c_5t^{-d_f/d_w}\exp\Big(-c_6\Big(\frac{d(x,y)^{d_w}}t\Big)^{1/(d_w-1)}\Big),
\label{subGS}
\end{eqnarray}
where $d_w\geq 2$. We write the Markov process on $G$ corresponding to \eqref{eq:geneBTM} as $\{X_t\}_t$ (where $x\sim y$ means $x$ and $y$ are connected by a bond), and write its transition density, defined as at \eqref{qtd}, as $p_t^\tau(x,y)$. Fix $h\in(0,\infty)$, $a>1$ and $\frac12<\eta$. For $N\in\mathbb{N}$, set
\[\mathcal{I}_N:=\left\{K^{(i)}_{a^{-\lfloor {\eta N}\rfloor}}: i\in \Lambda,\:K^{(i)}_{a^{-\lfloor {\eta N}\rfloor}}\cap B(0, h a^{N+1})\neq \emptyset\right\},\]
where $B(0,r)$ is a ball centered at $0$ with radius $r$ with respect to the distance $d$. For $x\in  B(0, h)$ and $n\in[a^N,a^{N+1})$, let $I_x^n$ be an element in $\mathcal{I}_N$ containing
$nx$. For $I\in\mathcal{I}_N$, define $V(I):=\sum_{y\in I}\tau_y$. The following assumption captures the key estimates we need to deduce a quantitative local limit theorem.

\begin{assumption}\label{assumpLCLT}
There exists $\theta\in (0, d_f/(2d_w))$ such that for each $\frac{d_f}{2}<\kappa<\eta d_f<d_f$, $h>0$ and $0<T_1<T_2<\infty$, the following hold $\prob\mbox{-a.s.}$:
\begin{align}
\sup_{I\in\mathcal{I}_N}\left|V(I)-\mu(I)\right|&\leq C(\tau) a^{\kappa N},\label{ass1-1}\\
\sup_{\substack{|x|<h,\\t\in[T_1,T_2],\\n\in[a^N,a^{N+1})}}\left| P^\tau_0\big[X_{n^{d_w}t} \in I_x^n\big] - \int_{n^{-1}I_x^n}p_{\mathrm{BM}}(t, 0,y) \, dy \right|&\leq  C(\tau)a^{-d_w\theta N},\label{ass1-2}\\
\sup_{\substack{|x|<h:\:nx\in G,\\t\in[T_1,T_2],\\n\in[a^N,a^{N+1})}} n^{d_f}\sup_{y\in I_x^n}\left| p^\tau_{n^{d_w}t}(0,y)- p^\tau_{n^{d_w}t}(0,nx)\right|
&\leq  C(\tau) a^{\frac12N(\eta-1)},\label{ass1-3}\\
\sup_{t\in[T_1,T_2]}\left| p_{\mathrm{BM}}(t, 0,y)- p_{\mathrm{BM}}(t, 0,x)\right| &\leq  Cd(x,y)^{\frac 12}.\label{ass1-4}
\end{align}
\end{assumption}

\begin{remark}
Note that, since $\mu$ has full support, it follows from \eqref{ass1-1} that for every $x\in M$, there exists a sequence $x_n\in G_n$, $n\geq 1$, such that $d(x_n,x)\rightarrow 0$.
\end{remark}

\begin{theorem}\label{thm:quant_LCLT*}
Under Assumption \ref{assumpLCLT}, it holds that for all $h>0$ and $0<T_1\leq T_2$ and any $\varepsilon>0$,
\begin{align*}
\lim_{n\to  \infty} n^{\frac{d_w\theta}{1+2d_f}-\varepsilon} \sup_{|x|\leq h:\:nx\in G} \sup_{t\in[T_1,T_2]}\big| n^{d_f} p^\tau_{n^{d_w}t}(0, nx)-p_{\mathrm{BM}}(t, 0,x) \big| =0,\qquad\prob\mbox{-a.s.}
\end{align*}
\end{theorem}
\begin{remark}
If $G_n:= b^{-n}G$ for $b>1$ and Assumption \ref{assumpLCLT} holds with $b^n$ instead of $n$, then Theorem~\ref{thm:quant_LCLT*} holds with $b^n$ instead of  $n$.
\end{remark}

\begin{proof}
Suppose $n\in [a^N,a^{N+1})$ and $x\in B(0,h)\cap n^{-1}G$. It then holds that
\begin{align*}
\lefteqn{\big|n^{d_f} p^\tau_{n^{d_w}t}(0, nx)-p_{\mathrm{BM}}(t, 0,x) \big|}\\
&\leq \Big|n^{d_f} p^\tau_{n^{d_w}t}(0, nx)-  \frac{n^{d_f}P^\tau_0[X_{n^{d_w}t} \in I_x^n]}{V(I_x^n)} \Big| \\
& \qquad+\Big| \frac{n^{d_f}}{V(I_x^n)}  \Big( P^\tau_0\big[X_{n^{d_w}t} \in I_x^n\big] - \int_{n^{-1}I_x^n} p_{\mathrm{BM}}(t, 0,y) \, dy \Big) \Big| \\
&\qquad + \frac{1}{n^{-d_f}\mu(I_x^n)}\int_{n^{-1}I_x^n} p_{\mathrm{BM}}(t, 0,y) \, \mu(dy) \cdot \Big|1 - \frac{\mu(I_x^n)}{V(I_x^n)} \Big| \\
& \qquad+\Big| \frac 1 {n^{-d_f}\mu(I_x^n)}   \int_{n^{-1}I_x^n} p_{\mathrm{BM}}(t, 0,y) \, \mu(dy)-p_{\mathrm{BM}}(t, 0,x) \Big|\\
&=:\Upsilon_1+\Upsilon_2+\Upsilon_3+\Upsilon_4.
\end{align*}
By \eqref{ass1-4},
\begin{eqnarray*}
\sup_{t\in[T_1,T_2]}\sup_{|x|<h}\Upsilon_4&\leq& \sup_{t\in[T_1,T_2]}\sup_{|x|<h}\sup_{y\in n^{-1}I_x^n}\left| p_{\mathrm{BM}}
(t, 0,y)- p_{\mathrm{BM}}(t, 0,x)\right|\\
&\leq &Cn^{-1/2} a^{\lfloor\eta N\rfloor/2}\\
&\leq &Cn^{\frac12(\eta-1)}.
\end{eqnarray*}
Similarly, by \eqref{ass1-3}, we have that
\begin{eqnarray*}
\sup_{\substack{|x|<h:\:nx\in G,\\t\in[T_1,T_2]}}\Upsilon_1&\leq&\sup_{\substack{|x|<h:\:nx\in G,\\t\in[T_1,T_2]}} n^{d_f}\sup_{y\in I_x^n}\left| p^\tau_{n^{d_w}t}(0,y)- p^\tau_{n^{d_w}t}(0,nx)\right|\leq Cn^{\frac12(\eta-1)}.
\end{eqnarray*}
As for $\Upsilon_2$, from \eqref{ass1-2}, we deduce that
\[\sup_{\substack{|x|<h,\\t\in[T_1,T_2]}}\left| P_0\big[X_{n^{d_w}t} \in I_x^n\big] - \int_{n^{-1}I_x^n} p_{\mathrm{BM}}(t, 0,y) \, \mu(dy) \right|\leq Cn^{-{d_w}\theta}.\]
From the volume bound \eqref{ass1-1}, we also have that
\[\sup_{|x|<h}\frac{n^{d_f}}{V(I_x^n)}\leq \frac{n^{d_f}}{a^{\lfloor\eta N\rfloor d_f}-a^{\kappa N}}\leq Cn^{d_f(1-\eta)}.\]
Combining the two previous bounds yields
\[\sup_{\substack{|x|<h,\\t\in[T_1,T_2]}}\Upsilon_2\leq  Cn^{d_f(1-\eta)-d_w\theta}.\]
Moreover, in the third term, $\Upsilon_3$, we have that
\[\sup_{\substack{|x|<h,\\t\in[T_1,T_2]}}\frac{1}{n^{-d_f}\mu(I_x^n)}\int_{n^{-1}I_x^n} p_{\mathrm{BM}}(t, 0,y) \, dy \leq\sup_{x\in M,\:t\in[T_1,T_2]} p_{\mathrm{BM}}(t, x,x)\leq C,\]
and also, again from \eqref{ass1-1},
\[\sup_{|x|<h}\Big|1 - \frac{\mu(I_x^n)}{V(I_x^n)} \Big|\leq Cn^{\kappa-\eta d_f}, \]
hence
\[\Upsilon_3\leq  Cn^{\kappa-\eta d_f}.\]
Putting these pieces together, we have established that
\begin{align*}
\lefteqn{\sup_{|x|\leq h:\:nx\in G} \sup_{t\in[T_1,T_2]}\big| n^{d_f} p^\omega_{n^{d_w}t}(0, nx)-p_{\mathrm{BM}}(t, 0,x) \big|}\\
&\leq C\left(n^{\frac12(\eta-1)}+n^{d_f(1-\eta)-d_w\theta}+n^{\kappa-\eta d_f}\right).
\end{align*}
Using that $\theta<d_f/(2d_w)$, it is an elementary exercise to check that this bound is optimised over the $\kappa$ and $\eta$ satisfying the constraints of Assumption \ref{assumpLCLT} by taking $\kappa$ suitably close to $d_f/2$ and $\eta=\frac{1+2d_f-2d_w\theta}{1+2d_f}$. With this choice, we obtain the desired result.
\end{proof}

We now apply the above result to our case. In particular, we let $M$ be given by $\mathbb R$ equipped with the Euclidean metric, $\mu$ be Lebesgue measure multiplied by $\mathbf{E}(\tau_0)$, $G=\mathbb Z$, $K^{(i)}$ be the unit interval $[i,i+1]$ and $V^{(i)}$ is its boundary, and $B$ be Brownian motion on $\mathbb{R}$, slowed down by a factor $\mathbf{E}(\tau_0)$. Clearly, in this setting, \eqref{subGS} is satisfied with $d_f=1$, $d_w=2$. Moreover, taking $a=2$ and $x_i:=i\lfloor 2^{\eta N}\rfloor$, $\mathcal{I}_N$ is given by
\[\mathcal{I}_N:=\left\{[x_{i-1},x_i]:\:x_i\geq -h2^{N+1},\:x_{i-1}\leq h2^{N+1}\right\}.\]

\begin{proof}[Proof of Theorem~\ref{thm:QLCLT}] It is enough to check Assumption \ref{assumpLCLT}. We first prove \eqref{ass1-1}. By a union bound, we have (for sufficiently large $N$) that
\begin{eqnarray*}
\mathbf{P}\left(\sup_{I\in\mathcal{I}_N}\left|V(I)-\mu(I)-\mathbf{E}(\tau_0)\right|> a^{\kappa N}\right)
&\leq &\left|\mathcal{I}_N\right|\frac{\mathrm{Var}\left(V(I)\right)}{a^{2\kappa N}}\\
&\leq & \frac{C a^N}{a^{\lfloor {\eta N}\rfloor}}\times
\frac{a^{\lfloor {\eta N}\rfloor}\mathrm{Var}\left(\tau_0\right)}{a^{2\kappa N}}\\
&=& Ca^{(1-2\kappa)N}.
\end{eqnarray*}
Thus the result is a simple application of the Borel-Cantelli lemma. The assumption at \eqref{ass1-2} is a consequence of Theorem~\ref{thm:be}\,(ii), where we may take $\theta$ as in that result. The bound at \eqref{ass1-3} follows from \eqref{eq:holder-HK}, and \eqref{ass1-4} is a standard estimate for $1$-dimensional Brownian motion (which can be proved similarly to \eqref{eq:holder-HK}). Hence, with some adjustment of the constants, we obtain the theorem.
\end{proof}

\appendix

\section{Tail bounds for the volume of balls in the BTM}\label{appa}

In this section, we present tail bounds for $S_n$, which imply corresponding bounds on the volume of balls in the Bouchaud trap model, as will be needed in the proof of Theorem~\ref{thm:AHK}. The results are similar to ones established in \cite{BC,Cline,Heyde,MN,Nag}, for example, but for the convenience of the reader we present self-contained proofs. The proof of part (ii) in particular is modelled on the argument of \cite[Proposition~2.1]{BC}. The definition of $v_\alpha$ should be recalled from \eqref{vdef}.

\begin{prop}\label{prop:tail}
\begin{enumerate}
  \item[(i)] For any $\alpha>0$, it holds that: for $n\in\mathbb{N}$ and $\lambda\geq 1$,
\[\prob\left(S_n\leq \lambda^{-1}v_\alpha(n)\right)\leq Ce^{-c\lambda^\gamma},\]
where we can take $\gamma=\min\{1,\alpha\}$.
  \item[(ii)] For any $\alpha>0$, it holds that: for $n\in\mathbb{N}$ and $\lambda\geq 1$,
\[\prob\left(S_n\geq \lambda v_\alpha(n)\right)\leq C\lambda^{-\rho},\]
where we can take $\rho$ to be any number strictly smaller that $\alpha$ (and the constant $C$ depends on $\rho$).
\end{enumerate}
\end{prop}
\begin{proof} Towards checking part~(i), we first note that, for $\theta<1$,
\[\mean\left(e^{-\theta \tau_1}\right)\leq \left\{
                                             \begin{array}{ll}
                                               1-c\theta^\alpha, & \hbox{if $\alpha<1$;} \\
                                               1-c\theta\log(\theta^{-1}), & \hbox{if $\alpha=1$;} \\
                                               1-c\theta, & \hbox{if $\alpha>1$.}
                                             \end{array}
                                           \right.\]
Hence, for $n\lambda\leq v_\alpha(n)$,
\begin{align*}
\prob\left(S_n\leq \lambda^{-1}v_\alpha(n)\right)
&\leq C\mean\left(e^{-\lambda S_n/v_\alpha(n)}\right)\\
&= C\mean\left(e^{-\lambda \tau_1/v_\alpha(n)}\right)^n\\
&\leq  \left\{
                                             \begin{array}{ll}
                                               C\left(1-\frac{c\lambda^\alpha}{n}\right)^n, & \hbox{if $\alpha<1$;} \\
                                               C\left(1-\frac{c\lambda}{n}\right)^n, & \hbox{if $\alpha\geq1$.}
                                             \end{array}
                                           \right.
\end{align*}
Since $1-x\leq e^{-x}$, the desired conclusion thus holds in the range $n\lambda<v_\alpha(n)$. If $n\lambda>v_\alpha(n)$, then it trivially holds that
$\prob\left(S_n\leq \lambda^{-1}v_\alpha(n)\right)\leq \prob\left(S_n<n\right)=0$, since each individual summand is bounded below by 1.

As for part~(ii), elementary calculations yield that, for $\theta<1$,
\[\mean\left(e^{\theta \tau_1}\indicator_{\{\tau_1\leq \theta^{-1}\}}\right)\leq \left\{
                                             \begin{array}{ll}
                                               1+c\theta^\alpha, & \hbox{if $\alpha<1$;} \\
                                               1+c\theta\log(\theta^{-1}), & \hbox{if $\alpha=1$;} \\
                                               1+c\theta, & \hbox{if $\alpha>1$.}
                                             \end{array}
                                           \right.\]
Hence, for $\alpha<1$, if $m=\lambda^\eta n^{1/\alpha}$ for some $\eta\in(0,1)$, then
\begin{align*}
\prob\left(S_n\geq \lambda v_\alpha(n)\right)
&\leq n\prob\left(\tau_1\geq m\right)+ \mean\left(e^{S_n/m}\indicator_{\{\tau_1,\dots,\tau_n\leq m\}}\right)e^{- \lambda v_\alpha(n)/m}\\
&\leq nm^{-\alpha}+\left(1+cm^{-\alpha}\right)^ne^{- \lambda v_\alpha(n)/m}\\
&\leq \lambda^{-\eta\alpha}+e^{c\lambda^{-\eta \alpha}- \lambda^{1-\eta}}\\
&\leq C\lambda^{-\eta\alpha}.
\end{align*}
This completes the proof in this case. Taking $m=\lambda^\eta n$ for some $\eta\in(0,1)$, the argument for $\alpha=1$ is similar, and so we omit it. For $\alpha>1$, set $m=\lambda^\eta n^{1/\alpha}$ for some $\eta\in(0,1)$, then
\begin{align*}
\prob\left(S_n\geq \lambda v_\alpha(n)\right)
&\leq nm^{-\alpha}+\left(1+cm^{-1}\right)^ne^{- \lambda v_\alpha(n)/m}\\
&\leq \lambda^{-\eta\alpha}+e^{-(\lambda -c)\lambda^{-\eta}n^{1-1/\alpha}}\\
&\leq C\lambda^{-\eta\alpha}+Ce^{-c \lambda^{1-\eta}}\\
&\leq C\lambda^{-\eta\alpha},
\end{align*}
which completes the proof.
\end{proof}

{\bf Acknowledgements}:
This research was supported by JSPS Grant-in-Aid for Scientific Research (A) 22H00099, 23KK0050, (C) 19K03540, and the Research Institute for Mathematical Sciences, an International Joint Usage/Research Center located in Kyoto University. The authors thank Ryoki Fukushima for useful comments concerning the proof of Proposition~\ref{prop4-3}\,(ii), as well as two referees for catching various minor errors that appeared in an earlier version.

\bibliographystyle{abbrv}
\bibliography{literature}

\begin{thebibliography}{10}

\bibitem{ACK}
S.~Andres, D.~Croydon, and T.~Kumagai.
\newblock Heat kernel fluctuations for stochastic processes on fractals and
  random media.
\newblock In {\em From classical analysis to analysis on fractals. {V}ol. 1.
  {A} tribute to {R}obert {S}trichartz}, Appl. Numer. Harmon. Anal., pages
  265--281. Birkh\"{a}user/Springer, Cham, [2023] \copyright 2023.

\bibitem{AN19}
S.~Andres and S.~Neukamm.
\newblock Berry-{E}sseen theorem and quantitative homogenization for the random
  conductance model with degenerate conductances.
\newblock {\em Stoch. Partial Differ. Equ. Anal. Comput.}, 7(2):240--296, 2019.

\bibitem{AG18}
S.~Armstrong and P.~Dario.
\newblock Elliptic regularity and quantitative homogenization on percolation
  clusters.
\newblock {\em Comm. Pure Appl. Math.}, 71(9):1717--1849, 2018.

\bibitem{AK22}
S.~Armstrong and T.~Kuusi.
\newblock Elliptic homogenization from qualitative to quantitative.
\newblock {\em Preprint, available at arXiv:2210.06488}, 2022.

\bibitem{armstrong2016b}
S.~Armstrong, T.~Kuusi, and J.-C. Mourrat.
\newblock {The additive structure of elliptic homogenization}.
\newblock {\em Invent. Math.}, 208(3):999--1154, 2017.

\bibitem{AKM19}
S.~Armstrong, T.~Kuusi, and J.-C. Mourrat.
\newblock {\em Quantitative stochastic homogenization and large-scale
  regularity}, volume 352 of {\em Grundlehren der mathematischen Wissenschaften
  [Fundamental Principles of Mathematical Sciences]}.
\newblock Springer, Cham, 2019.

\bibitem{armstrong2016}
S.~N. {Armstrong} and J.-C. {Mourrat}.
\newblock {Lipschitz Regularity for Elliptic Equations with Random
  Coefficients}.
\newblock {\em Archive for Rational Mechanics and Analysis}, 219:255--348, Jan.
  2016.

\bibitem{armstrong2014}
S.~N. Armstrong and C.~K. Smart.
\newblock {Quantitative stochastic homogenization of convex integral
  functionals}.
\newblock {\em Ann. Sci. {\'E}c. Norm. Sup{\'e}r. (4)}, 49(2):423--481, 2016.

\bibitem{Bar}
M.~T. Barlow.
\newblock Diffusions on fractals.
\newblock In {\em Lectures on probability theory and statistics
  ({S}aint-{F}lour, 1995)}, volume 1690 of {\em Lecture Notes in Math.}, pages
  1--121. Springer, Berlin, 1998.

\bibitem{Barbook}
M.~T. Barlow.
\newblock {\em Random walks and heat kernels on graphs}, volume 438 of {\em
  London Mathematical Society Lecture Note Series}.
\newblock Cambridge University Press, Cambridge, 2017.

\bibitem{BCK}
M.~T. Barlow, T.~Coulhon, and T.~Kumagai.
\newblock Characterization of sub-{G}aussian heat kernel estimates on strongly
  recurrent graphs.
\newblock {\em Comm. Pure Appl. Math.}, 58(12):1642--1677, 2005.

\bibitem{BJKS}
M.~T. Barlow, A.~A. J\'{a}rai, T.~Kumagai, and G.~Slade.
\newblock Random walk on the incipient infinite cluster for oriented
  percolation in high dimensions.
\newblock {\em Comm. Math. Phys.}, 278(2):385--431, 2008.

\bibitem{BK06}
M.~T. Barlow and T.~Kumagai.
\newblock Random walk on the incipient infinite cluster on trees.
\newblock {\em Illinois J. Math.}, 50(1-4):33--65, 2006.

\bibitem{BellaFFOtto2016}
P.~Bella, B.~Fehrman, J.~Fischer, and F.~Otto.
\newblock {Stochastic homogenization of linear elliptic equations: higher-order
  error estimates in weak norms via second-order correctors}.
\newblock {\em SIAM J. Math. Anal.}, 49(6):4658--4703, 2017.

\bibitem{BellaFO2016}
P.~Bella, B.~Fehrman, and F.~Otto.
\newblock {A {L}iouville theorem for elliptic systems with degenerate ergodic
  coefficients}.
\newblock {\em Ann. Appl. Probab.}, 28(3):1379--1422, 2018.

\bibitem{BCCR}
G.~Ben~Arous, M.~Cabezas, J.~\v{C}ern\'{y}, and R.~Royfman.
\newblock Randomly trapped random walks.
\newblock {\em Ann. Probab.}, 43(5):2405--2457, 2015.

\bibitem{BC05}
G.~Ben~Arous and J.~\v{C}ern\'{y}.
\newblock Bouchaud's model exhibits two different aging regimes in dimension
  one.
\newblock {\em Ann. Appl. Probab.}, 15(2):1161--1192, 2005.

\bibitem{BC06}
G.~Ben~Arous and J.~\v{C}ern\'{y}.
\newblock Dynamics of trap models.
\newblock In {\em Mathematical statistical physics}, pages 331--394. Elsevier
  B. V., Amsterdam, 2006.

\bibitem{BJ73}
A.~Benveniste and J.~Jacod.
\newblock Syst\`emes de {L}\'{e}vy des processus de {M}arkov.
\newblock {\em Invent. Math.}, 21:183--198, 1973.

\bibitem{BGT}
N.~H. Bingham, C.~M. Goldie, and J.~L. Teugels.
\newblock {\em Regular variation}, volume~27 of {\em Encyclopedia of
  Mathematics and its Applications}.
\newblock Cambridge University Press, Cambridge, 1989.

\bibitem{Bo79}
A.~N. Borodin.
\newblock A limit theorem for sums of independent random variables defined on a
  recurrent random walk.
\newblock {\em Dokl. Akad. Nauk SSSR}, 246(4):786--787, 1979.

\bibitem{Bo79a}
A.~N. Borodin.
\newblock Limit theorems for sums of independent random variables defined on a
  transient random walk.
\newblock {\em Zap. Nauchn. Sem. Leningrad. Otdel. Mat. Inst. Steklov. (LOMI)},
  pages 17--29, 237, 244, 1979.
\newblock Investigations in the theory of probability distributions, IV.

\bibitem{Bou92}
J.~P. Bouchaud.
\newblock Weak ergodicity breaking and aging in disordered systems.
\newblock {\em J.\ Phys.\ I France}, 2(9):1705 -- 1713, 1992.

\bibitem{BC}
A.~M. Bowditch and D.~A. Croydon.
\newblock Biased random walk on supercritical percolation: anomalous
  fluctuations in the ballistic regime.
\newblock {\em Electron. J. Probab.}, 27:Paper No. 68, 22, 2022.

\bibitem{Ca15}
M.~Cabezas.
\newblock Sub-{G}aussian bound for the one-dimensional {B}ouchaud trap model.
\newblock {\em Braz. J. Probab. Stat.}, 29(1):112--131, 2015.

\bibitem{Ce06}
J.~{\v C}ern\'{y}.
\newblock The behaviour of aging functions in one-dimensional {B}ouchaud's trap
  model.
\newblock {\em Comm. Math. Phys.}, 261(1):195--224, 2006.

\bibitem{CK03}
Z.-Q. Chen and T.~Kumagai.
\newblock Heat kernel estimates for stable-like processes on {$d$}-sets.
\newblock {\em Stochastic Process. Appl.}, 108(1):27--62, 2003.

\bibitem{Cline}
D.~B.~H. Cline and T.~Hsing.
\newblock Large deviation probabilities for sums and maxima of random variables
  with heavy or subexponential tails.
\newblock Preprint, 1998.

\bibitem{Croy}
D.~A. Croydon.
\newblock Heat kernel fluctuations for a resistance form with non-uniform
  volume growth.
\newblock {\em Proc. Lond. Math. Soc. (3)}, 94(3):672--694, 2007.

\bibitem{Croyres}
D.~A. Croydon.
\newblock Scaling limits of stochastic processes associated with resistance
  forms.
\newblock {\em Ann. Inst. Henri Poincar\'{e} Probab. Stat.}, 54(4):1939--1968,
  2018.

\bibitem{CHLLT}
D.~A. Croydon and B.~M. Hambly.
\newblock Local limit theorems for sequences of simple random walks on graphs.
\newblock {\em Potential Anal.}, 29(4):351--389, 2008.

\bibitem{CHKres}
D.~A. Croydon, B.~M. Hambly, and T.~Kumagai.
\newblock Time-changes of stochastic processes associated with resistance
  forms.
\newblock {\em Electron. J. Probab.}, 22:Paper No. 82, 41, 2017.

\bibitem{CHK}
D.~A. Croydon, B.~M. Hambly, and T.~Kumagai.
\newblock Heat kernel estimates for {FIN} processes associated with resistance
  forms.
\newblock {\em Stochastic Process. Appl.}, 129(9):2991--3017, 2019.

\bibitem{CKtree}
D.~A. Croydon and T.~Kumagai.
\newblock Random walks on {G}alton-{W}atson trees with infinite variance
  offspring distribution conditioned to survive.
\newblock {\em Electron. J. Probab.}, 13:no. 51, 1419--1441, 2008.

\bibitem{CM15}
D.~A. Croydon and S.~Muirhead.
\newblock Functional limit theorems for the {B}ouchaud trap model with slowly
  varying traps.
\newblock {\em Stochastic Process. Appl.}, 125(5):1980--2009, 2015.

\bibitem{CM16}
D.~A. Croydon and S.~Muirhead.
\newblock Quenched localisation in the {B}ouchaud trap model with regularly
  varying traps.
\newblock In {\em Stochastic analysis on large scale interacting systems},
  volume B59 of {\em RIMS K\^{o}ky\^{u}roku Bessatsu}, pages 305--320. Res.
  Inst. Math. Sci. (RIMS), Kyoto, 2016.

\bibitem{CM17}
D.~A. Croydon and S.~Muirhead.
\newblock Quenched localisation in the {B}ouchaud trap model with slowly
  varying traps.
\newblock {\em Probab. Theory Related Fields}, 168(1-2):269--315, 2017.

\bibitem{Cshir}
D.~A. Croydon and D.~Shiraishi.
\newblock Scaling limit for random walk on the range of random walk in four
  dimensions.
\newblock {\em Ann. Inst. Henri Poincar\'{e} Probab. Stat.}, 59(1):166--184,
  2023.

\bibitem{DaGu}
P.~Dario and C.~Gu.
\newblock Quantitative homogenization of the parabolic and elliptic {G}reen's
  functions on percolation clusters.
\newblock {\em Ann. Probab.}, 49(2):556--636, 2021.

\bibitem{DF19}
J.-D. Deuschel and R.~Fukushima.
\newblock Quenched tail estimate for the random walk in random scenery and in
  random layered conductance.
\newblock {\em Stochastic Process. Appl.}, 129(1):102--128, 2019.

\bibitem{DF20}
J.-D. Deuschel and R.~Fukushima.
\newblock Quenched tail estimate for the random walk in random scenery and in
  random layered conductance {II}.
\newblock {\em Electron. J. Probab.}, 25:Paper No. 75, 28, 2020.

\bibitem{DGO20}
M.~Duerinckx, A.~Gloria, and F.~Otto.
\newblock The structure of fluctuations in stochastic homogenization.
\newblock {\em Comm. Math. Phys.}, 377(1):259--306, 2020.

\bibitem{EL}
U.~Einmahl and D.~Li.
\newblock Some results on two-sided {LIL} behavior.
\newblock {\em Ann. Probab.}, 33(4):1601--1624, 2005.

\bibitem{Fell}
W.~Feller.
\newblock An extension of the law of the iterated logarithm to variables
  without variance.
\newblock {\em J. Math. Mech.}, 18:343--355, 1968/69.

\bibitem{FIN99}
L.~R.~G. Fontes, M.~Isopi, and C.~M. Newman.
\newblock Chaotic time dependence in a disordered spin system.
\newblock {\em Probab. Theory Related Fields}, 115(3):417--443, 1999.

\bibitem{FIN}
L.~R.~G. Fontes, M.~Isopi, and C.~M. Newman.
\newblock Random walks with strongly inhomogeneous rates and singular
  diffusions: convergence, localization and aging in one dimension.
\newblock {\em Ann. Probab.}, 30(2):579--604, 2002.

\bibitem{GNO15}
A.~Gloria, S.~Neukamm, and F.~Otto.
\newblock {Quantification of ergodicity in stochastic homogenization: optimal
  bounds via spectral gap on {G}lauber dynamics}.
\newblock {\em Invent. Math.}, 199(2):455--515, 2015.

\bibitem{GNOreg}
A.~Gloria, S.~Neukamm, and F.~Otto.
\newblock A regularity theory for random elliptic operators.
\newblock {\em Milan J. Math.}, 88(1):99--170, 2020.

\bibitem{GO11}
A.~Gloria and F.~Otto.
\newblock {An optimal variance estimate in stochastic homogenization of
  discrete elliptic equations}.
\newblock {\em Ann. Probab.}, 39(3):779--856, 2011.

\bibitem{GO12}
A.~Gloria and F.~Otto.
\newblock {An optimal error estimate in stochastic homogenization of discrete
  elliptic equations}.
\newblock {\em Ann. Appl. Probab.}, 22(1):1--28, 2012.

\bibitem{GO15}
A.~{Gloria} and F.~{Otto}.
\newblock {The corrector in stochastic homogenization: optimal rates,
  stochastic integrability, and fluctuations}.
\newblock {\em ArXiv e-prints}, Oct. 2015.

\bibitem{GO14}
A.~Gloria and F.~Otto.
\newblock {Quantitative results on the corrector equation in stochastic
  homogenization}.
\newblock {\em J. Eur. Math. Soc. (JEMS)}, 19(11):3489--3548, 2017.

\bibitem{GK68}
B.~V. Gnedenko and A.~N. Kolmogorov.
\newblock {\em Limit distributions for sums of independent random variables}.
\newblock Addison-Wesley Publishing Co., Reading, Mass.-London-Don Mills, Ont.,
  revised edition, 1968.
\newblock Translated from the Russian, annotated, and revised by K. L. Chung,
  With appendices by J. L. Doob and P. L. Hsu.

\bibitem{Ha88}
E.~Haeusler.
\newblock {On the rate of convergence in the central limit theorem for
  martingales with discrete and continuous time}.
\newblock {\em Ann. Probab.}, 16(1):275--299, 1988.

\bibitem{HW}
P.~Hartman and A.~Wintner.
\newblock On the law of the iterated logarithm.
\newblock {\em Amer. J. Math.}, 63:169--176, 1941.

\bibitem{Heyde}
C.~C. Heyde.
\newblock On large deviation problems for sums of random variables which are
  not attracted to the normal law.
\newblock {\em Ann. Math. Statist.}, 38:1575--1578, 1967.

\bibitem{HB70}
C.~C. Heyde and B.~M. Brown.
\newblock {On the departure from normality of a certain class of martingales}.
\newblock {\em Ann. Math. Statist.}, 41:2161--2165, 1970.

\bibitem{KS79}
H.~Kesten and F.~Spitzer.
\newblock A limit theorem related to a new class of self-similar processes.
\newblock {\em Z. Wahrsch. Verw. Gebiete}, 50(1):5--25, 1979.

\bibitem{Kl1}
M.~J. Klass.
\newblock Toward a universal law of the iterated logarithm. {I}.
\newblock {\em Z. Wahrscheinlichkeitstheorie und Verw. Gebiete},
  36(2):165--178, 1976.

\bibitem{Kl2}
M.~J. Klass.
\newblock Toward a universal law of the iterated logarithm. {II}.
\newblock {\em Z. Wahrscheinlichkeitstheorie und Verw. Gebiete},
  39(2):151--165, 1977.

\bibitem{Kum}
T.~Kumagai.
\newblock Heat kernel estimates and parabolic {H}arnack inequalities on graphs
  and resistance forms.
\newblock {\em Publ. Res. Inst. Math. Sci.}, 40(3):793--818, 2004.

\bibitem{KumSF}
T.~Kumagai.
\newblock {\em Random walks on disordered media and their scaling limits},
  volume 2101 of {\em Lecture Notes in Mathematics}.
\newblock Springer, Cham, 2014.
\newblock Lecture notes from the 40th Probability Summer School held in
  Saint-Flour, 2010, \'{E}cole d'\'{E}t\'{e} de Probabilit\'{e}s de
  Saint-Flour. [Saint-Flour Probability Summer School].

\bibitem{KM}
T.~Kumagai and J.~Misumi.
\newblock Heat kernel estimates for strongly recurrent random walk on random
  media.
\newblock {\em J. Theoret. Probab.}, 21(4):910--935, 2008.

\bibitem{Ma94}
D.~M. Mason.
\newblock A universal one-sided law of the iterated logarithm.
\newblock {\em Ann. Probab.}, 22(4):1826--1837, 1994.

\bibitem{MN}
T.~Mikosch and A.~V. Nagaev.
\newblock Large deviations of heavy-tailed sums with applications in insurance.
\newblock {\em Extremes}, 1(1):81--110, 1998.

\bibitem{Mo12}
J.-C. Mourrat.
\newblock {A quantitative central limit theorem for the random walk among
  random conductances}.
\newblock {\em Electron. J. Probab.}, 17:no. 97, 17, 2012.

\bibitem{Mo12a}
J.-C. Mourrat.
\newblock {On the rate of convergence in the martingale central limit theorem}.
\newblock {\em Bernoulli}, 19(2):633--645, 2013.

\bibitem{Mu15}
S.~Muirhead.
\newblock Two-site localisation in the {B}ouchaud trap model with slowly
  varying traps.
\newblock {\em Electron. Commun. Probab.}, 20:no. 25, 15, 2015.

\bibitem{Nag}
S.~V. Nagaev.
\newblock Large deviations of sums of independent random variables.
\newblock {\em Ann. Probab.}, 7(5):745--789, 1979.

\bibitem{Pr81}
W.~E. Pruitt.
\newblock General one-sided laws of the iterated logarithm.
\newblock {\em Ann. Probab.}, 9(1):1--48, 1981.

\bibitem{RW00}
L.~C.~G. Rogers and D.~Williams.
\newblock {\em Diffusions, {M}arkov processes, and martingales. {V}ol. 2}.
\newblock Cambridge Mathematical Library. Cambridge University Press,
  Cambridge, 2000.
\newblock It\^{o} calculus, Reprint of the second (1994) edition.

\bibitem{Ta96}
M.~Talagrand.
\newblock A new look at independence.
\newblock {\em Ann. Probab.}, 24(1):1--34, 1996.

\end{thebibliography}

\end{document}